\documentclass[a4paper, 10pt, twoside, notitlepage]{amsart}
\usepackage[utf8]{inputenc}
\usepackage{color}
\usepackage{amsmath} 
\usepackage{amssymb} 
\usepackage{amsthm}
\usepackage{graphicx}
\usepackage{esint}
\usepackage[colorlinks=true,linkcolor=blue]{hyperref}
\usepackage[export]{adjustbox}
\usepackage{thmtools}
\usepackage{subcaption}
\usepackage{hyphenat}

\usepackage{enumitem}

\theoremstyle{plain}
\newtheorem{thm}{Theorem}
\newtheorem{prop}{Proposition}[section]

\newtheorem{lem}[prop]{Lemma}
\newtheorem{cor}[prop]{Corollary}

\newtheorem{defi}[prop]{Definition}
\newtheorem{rmk}[prop]{Remark}

\newtheorem{alg}[prop]{Algorithm}

\newcommand {\R} {\mathbb{R}} \newcommand {\Z} {\mathbb{Z}}
 \newcommand {\N} {\mathbb{N}}
  \newcommand {\E} {\mathbb{E}}
 \newcommand {\F} {\mathcal{F}}

\newcommand {\D} {\Delta}

\newcommand{\vc}[1]{\mathbf{#1}}

\newcommand{\mt}[1]{\mathsf{#1}}

\DeclareMathOperator{\diag}{diag}

\DeclareMathOperator {\dist} {dist}

\DeclareMathOperator {\inte} {int}
\DeclareMathOperator {\argmin} {argmin}

\DeclareMathOperator {\Per} {Per}

\begin{document}

\title[]{On a probabilistic model for martensitic avalanches incorporating mechanical compatibility}

\author[F. Della Porta]{Francesco Della Porta}
\address{Max-Planck Institute for Mathematics in the Sciences, Inselstraße 22, 04103 Leipzig, Germany}
\email{DellaPorta@mis.mpg.de}
\author[A. Rüland]{Angkana Rüland}
\address{Ruprecht-Karls-Universität Heidelberg, Institut für Angewandte Mathematik, Im Neuenheimer Feld 294,
69120 Heidelberg, Germany}
\email{Angkana.Rueland@uni-heidelberg.de}
\author[J. Taylor]{Jamie M Taylor}
\address{Basque Center for Applied Mathematics, Mazarredo, 14, 48009 Bilbao Basque Country - Spain}
\email{taylor@bcamath.org}
\author[C. Zillinger]{Christian Zillinger}
\address{Karlsruhe Institute of Technology, Englerstraße 2,
76131 Karlsruhe, Germany}
\email{christian.zillinger@kit.edu}

\begin{abstract}
Building on the work in \cite{BCH15,CH18,TIVP17}, in this article we propose and study a simple, geometrically constrained, probabilistic algorithm geared towards capturing some aspects of the nucleation in shape-memory alloys. As a main novelty with respect to the algorithms in \cite{BCH15,CH18,TIVP17} we include \emph{mechanical compatibility}. The mechanical compatibility here is guaranteed by using \emph{convex integration building blocks} in the nucleation steps. We analytically investigate the algorithm's convergence and the solutions' regularity, viewing the latter as a measure for the fractality of the resulting microstructure. We complement our analysis with a numerical implemenation of the scheme and compare it to the numerical results in \cite{BCH15,CH18,TIVP17}.
\end{abstract}

\maketitle

\section{Introduction}

Shape-memory alloys are materials displaying a striking thermodynamical behaviour on the one hand and a rich mathematical structure on the other hand. Physically, these materials undergo a first-order, diffusionless, solid-solid phase transformation in which symmetry is reduced upon the passage from the high temperature phase, \emph{austenite}, to the low temperature phase, \emph{martensite}. This reduction of symmetry gives rise to various \emph{variants of martensite} in the low temperature regime. 

Mathematically, these materials have been successfully described within the calculus of variations by minimization problems of the form 
\begin{align}
\label{eq:min}
\mbox{ minimize } \int\limits_{\Omega} W(\nabla \vc y, \theta) d x,
\end{align}
for instance, with prescribed displacement boundary conditions \cite{B04, BJ92, B2, B3, B, M1}. Here $\Omega \subset \R^n$ is the reference configuration, $\theta: [0,\infty) \rightarrow [0,\infty)$ denotes temperature and $\vc y: \Omega \rightarrow \R^n$ is the deformation describing how the reference configuration is deformed. Defining $\R^{n\times n}_{+}$ to be the set of $n$ by $n$ matrices with positive determinant, the stored energy function $W: \R^{n\times n}_{+} \times [0,\infty) \rightarrow \R_+$ describes the energetic cost of a deformation at a given temperature. Physical requirements on it are \emph{frame indifference}, i.e. the fact that
\begin{align*}
W(\mt F) = W(\mt Q \mt F) \mbox{ for all } \mt F \in \R^{n \times n}_+, \ \mt Q \in SO(n),
\end{align*}
and \emph{material symmetry}, i.e. the fact that
\begin{align*}
W(\mt F) = W(\mt F \mt H) \mbox{ for all } \mt F \in \R^{n\times n}_+, \ \mt H \in \mathcal{P}, 
\end{align*}
where $\mathcal{P}$ denotes 
the (discrete) symmetry group which sends the austenite lattice into itself.

In particular, the zero set -- or physically the set of \emph{exactly stress-free strains} -- associated with $W$ is typically of the form

\begin{align*}
K(\theta) = \left\{ \begin{array}{ll}
\alpha(\theta) SO(n) Id & \mbox{ for } \theta > \theta_c, \\
\bigcup\limits_{m=1}^{N} SO(n)\mt U_j(\theta_c) \cup \alpha(\theta_c) SO(n) Id & \mbox{ for } \theta = \theta_c,\\
\alpha(\theta)\bigcup\limits_{m=1}^{N} SO(n)\mt U_j(\theta) & \mbox{ for } \theta \leq \theta_c.
\end{array} \right.
\end{align*}
Here the matrices $\mt U_j(\theta) \subset \R^{n\times n}_+$ are obtained through conjugation of $\mt U_1(\theta)$ by elements from $\mathcal{P}$ and represent the $N$ variants of martensite, while $\alpha(\theta):\R_+\to\R_+$ models the thermal expansion of the underlying lattice depending on temperature, with the convention that $\alpha(\theta_c)=1$. In order to study low energy configurations of \eqref{eq:min}, a common strategy \cite{B3, B, CDK, DM2, K1, K, R16, Rue16b, TS} is to first study exactly stress-free deformations by investigating the differential inclusion 
\begin{align}
\label{eq:diff_incl}
\nabla \vc y \in K(\theta).
\end{align}

While the study of the minimization problem \eqref{eq:min} has proved very successful and influential, e.g. in predicting interfaces between variants of martensite and scaling laws \cite{BJ92, KM1, KKO13, KK, C1, CO, CO1, Rue16b}, 
it is often the case that the dynamics of the phase transition play an important role in the formation of the complex microstructures observed in experiments (see e.g., \cite{SCDSJ13} and related comments in \cite{FDP1}). Indeed, observed microstructures are often the result of different smaller microstructures, nucleating at different  points of the domain and expanding. 
In order to preserve continuity of the deformation, and hence compatibility, these microstructures become finer and more complex when they encounter. However, they do not globally minimise an energy functional penalising interfaces between martensitic junctions. 
Such complex evolution has been observed both with optical microscopy, a common tool to analyse martensitic microstructure, and by phonon emission measurements, a second method based on the observation that every nucleation event is accompanied by an acoustic emission (see e.g., \cite{BBGVBZ}). In particular, both methods have thus been used for tracking the dynamics of nucleation phenomena. High time resolution measurements of the described type display strongly intermittend behaviour and the presence of ``avalanches'' \cite{PMV13, SKRSP09} with ``universal'', power law behaviour for central statistical quantities. 
Based on these and related observations, it has been the objective of several recent works to study simplified dynamic models of phase transformations in shape-memory alloys: On the one hand, the continuum mechanical models in \cite{FDP1,FDP2} seek to capture the evolution of the microstructures and the mechanical effects that the dynamics may have on them based on \emph{optical} microscopy observations. On the other hand, in parallel, simplified probabilistic, geometrically constrained dynamic models have been proposed and investigated in the literature \cite{BCH15, CH18, TIVP17} -- both in the mathematical and the physics community. The latter aim at predicting the above mentioned \emph{acoustic} observations and at deriving an improved understanding of the ``universal'', power law behavior for central statistical quantities. 
As shown in \cite[Figure 1]{TIVP17}, these  probabilistic, geometrically constrained models sometimes also successfully reflect the ``wild'', ``random'', irregular microstructures observed in optical microscopy.

In order to capture the avalanching phase transformation dynamics, the models proposed in \cite{BCH15, CH18, TIVP17} take into account two key features which are believed to be characteristic of many martensitic phase transformations:
\begin{itemize}
\item During the phase transformation a domain which has transformed from austenite to martensite does \emph{not} transform back (see also the moving mask hypotheses in \cite{FDP1}).
\item The nucleation domains are given by long (needle-like) domains (``plates'') which are oriented according to the  rank-one connections which are present between the wells (see e.g., the experimetal results in \cite{IHM13} and c.f. once more the moving mask hypotheses in \cite{FDP1}).
\end{itemize}
Based on this, the models in \cite{BCH15, CH18, TIVP17} roughly propose the following simplified, geometrically constrained nucleation mechanisms:
\begin{itemize}
\item[(i)] Choose a point randomly out of the sample/reference configuration and choose a direction (out of the possible rank-one directions, i.e. out of the directions of compatibility between austenite and a martensitic plate) randomly.
\item[(ii)] Nucleate a martensitic plate in the chosen direction through the chosen point until it hits another plate or the boundary of the sample.
\item[(iii)] Iterate this.
\end{itemize}

We emphasize that this leads to a purely ``scalar'' model which is not formulated on the level of the deformation gradients and, in particular, does \emph{not} take into account any compatibility of the associated deformation gradients beyond the fact that the nucleated plates should roughly be aligned with the rank-one directions.
Numerical simulations of these dynamics lead to highly fractal, self-organized, ``wild'' structures in the martensitic materials.
Based on the described dynamics, in their analysis and simulations in \cite{BCH15, CH18, TIVP17} the authors derive properties of the statistical distribution of martensitic plates and deduce self-similarity and power law behaviour in certain regimes. This may indicate that, in spite of the drastic simplifications, the geometrically imposed constraints could indeed provide insights into the experimentally measured universal exponents in the nucleation experiments.

It is the objective of this article, to propose and investigate an \emph{intermediate} model, capturing both the random, geometrically constrained, self-organizing behaviour discussed in \cite{BCH15, CH18, TIVP17} and including the key mechanical aspect of \emph{compatibility}. In the previous works the latter had only been taken into account in terms of fixing the orientation of the martensitic plate and not in terms of the associated deformation. As in \cite{BCH15, CH18, TIVP17} we are also interested in studying the universality properties of solutions. In contrast to these results we however focus on the \emph{regularity of solutions} as a measurement for this and the ``wildness'' of the microstructure, and interpret regularity as the main quantity from which statistical properties could be deduced (see Section \ref{sec:turb} below).
We further link this to our recent investigation of deterministic ``wild'' microstructures obtained through the method of convex integration.

\subsection{The model and the main results}

In the sequel, as a model setting, we focus on the geometrically non-linear, two-dimensional two-well problem. Extensions to other models would not pose any difficulties as pointed out in our discussion below. Fixing temperature below the transformation temperature, we thus consider
\begin{equation}
\label{defK_intro}
K = SO(2)\mt F_0\cup SO(2)\mt F_0^{-1}
\end{equation}
where $\mt F_0,\mt F_0^{-1}\in\R^{2\times2}$ are respectively given by
\begin{align}
\label{eq:matrices_intro}
{\mt F}_0 = \left[\begin{array}{ ccc } 1 & \gamma \\ 0 & 1  \end{array}\right],\qquad
{\mt F}_0^{-1} = \left[\begin{array}{ ccc } 1 & -\gamma \\ 0 & 1  \end{array}\right],
\end{align}
and $\gamma>0.$ 

\begin{figure}
\includegraphics[width = 0.5 \textwidth,page=1]{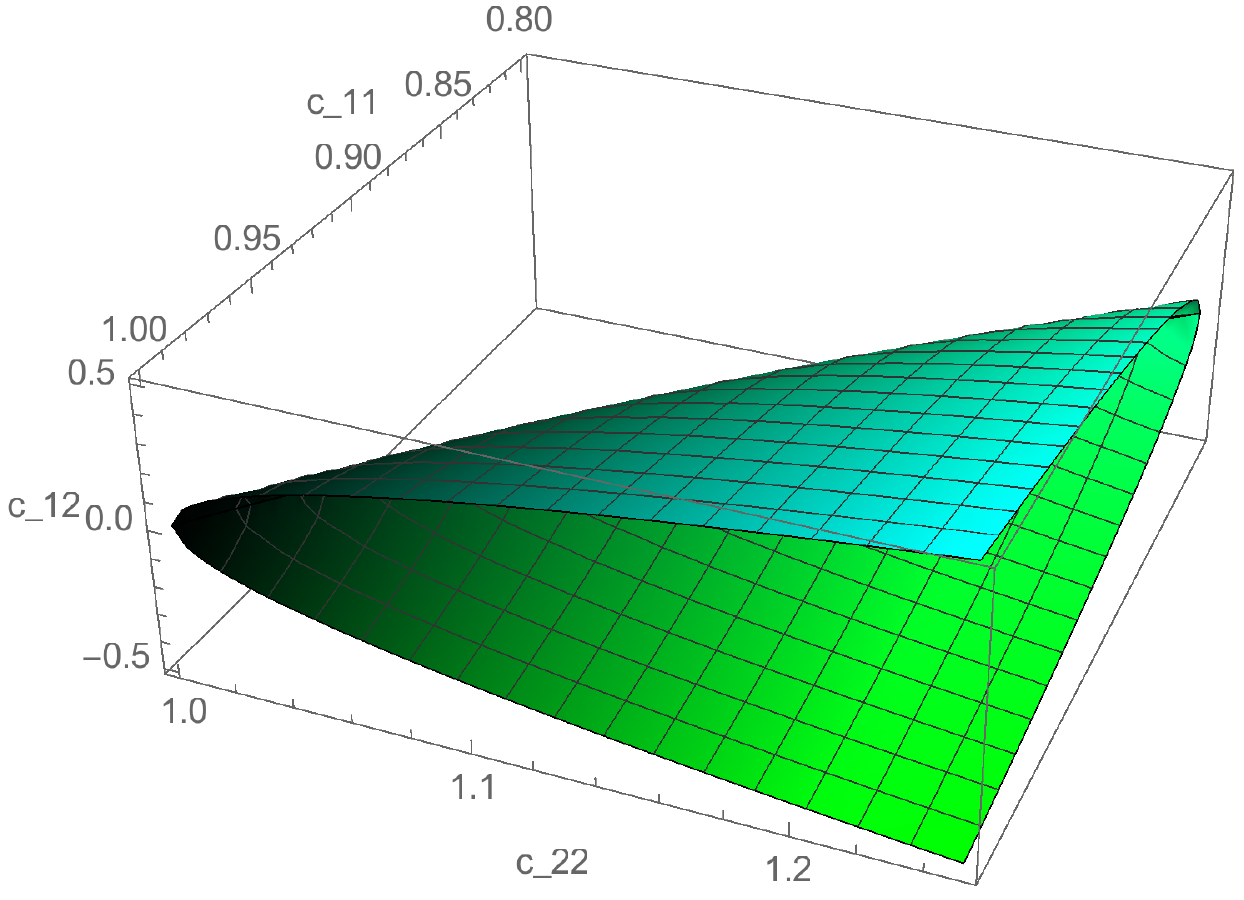}
\caption{The quasiconvex hull $K^{qc}$, i.e. the set of all macroscopically realizable deformations, associated with the set $K$, depicted in Cauchy-Green space, see Section \ref{sec:models} for more detailed definitions. Here the wells from \eqref{defK_intro} correspond to the two corners of the paraboloid (and are coloured cyan and green, respectively). All other matrices in the depicted set are obtained as Cauchy-Green tensors of first or second order laminates (corresponding to the boundaries of the paraboloid and its relative interior, respectively). The colour coding here is the colour coding which is used for vertical twins (see the explanations below). In order to illustrate the difference between horizontal and vertical twins (in the sense of \cite{DPR}), we use a second colour scheme (see Figure \ref{fig:vert}).}
\label{fig:Kqc}
\end{figure}

In the present work we propose dynamics which are strongly inspired by the ones in \cite{BCH15, CH18, TIVP17} but which take \emph{mechanical compatibility} into account. More precisely, essentially our nucleation algorithms still follow the steps (i)-(iii) from above, with the main difference that condition (ii) is now formulated on the level of the full deformation gradients (instead of the scalar order parameters from \cite{CH18, TIVP17}). Therefore, the plates which are nucleated in (ii) are now prescribed in a \emph{compatible} way (in the sense of not creating any stresses). This is achieved by relying on convex integration building blocks which are exactly stress-free solutions to the differential inclusion \eqref{eq:diff_incl} at a fixed temperature (in our case below the critical temperature) and with prescribed displacement boundary conditions. As in \cite{BCH15, CH18, TIVP17} for the setting of the two-well problem this gives rise to two specific orientations of the martensitic plates which are however now \emph{exact solutions} to the differential inclusion. In the infinite iteration/time limit, we thus obtain exactly stress-free solutions resembling those of \cite{BCH15, CH18, TIVP17} which now however are defined on the level of the deformations and in particular include compatibility and (up to a set of measure zero) fully transform $\Omega$. The precise algorithms used in our dynamics are described in Algorithms \ref{ModelA} and \ref{ModelB} in Section \ref{sec:models} below. 
Let us remark that such a behaviour reminds of experimental observations in TiNbAl (see e.g., \cite{Inamura}) where, after the phase transition, it is possible to observe different colonies of ``wild'' microstructures.

As in \cite{BCH15, CH18, TIVP17}, we seek to show that these dynamics give rise to ``power-law'' behaviour and self-organized structures in a probabilistic sense (see Section \ref{sec:sim_impl} for some numerical evaluations of the length scale statistics). As already indicated above, we do not aim at proving direct power-law distributions for the present lengths scales but view the \emph{regularity} of solutions as a proxy for this which encodes important statistical information (e.g. in terms of the solutions' heavy tailed Fourier distribution etc).

\begin{figure}[t]
\includegraphics[width=0.4 \textwidth]{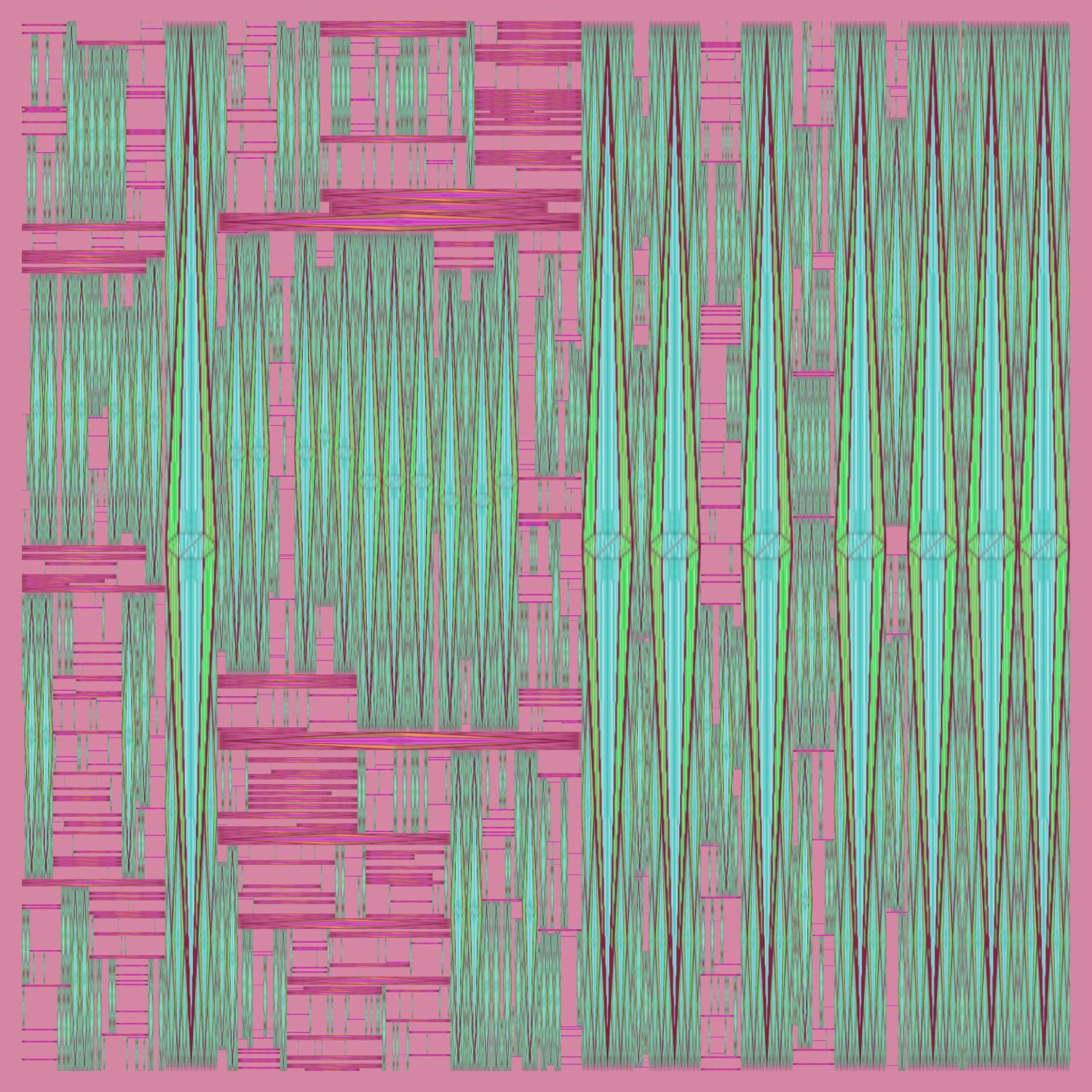}
\label{fig:ModelA2}
\caption{The radom convex integration solutions produced by (a variant) of our algorithms (the picture here is generated by means of the modifications explained in Section \ref{sec:sim_strat}). The colour coding uses cyan and green for vertical and magenta and orange for horizontal twins. In addition to the fractal behavior originating from the random (greedy type) covering which is observed in \cite{BCH15, CH18, TIVP17}, we here have a second source of fractality originating from the use of the convex integration building blocks within the random rectangle covering (see Section \ref{sec:models} below for more comments on this).}
\end{figure}

Our main analytical result for these dynamics is summarized in the following theorem:

\begin{thm}
\label{thm:main}
Let $K$ be as in \eqref{defK_intro}--\eqref{eq:matrices_intro} and let $\tilde\Omega =(0,1)^2 \subset \R^2$ 
Let $\{\vc y_k\}$ denote the sequences obtained in Algorithms \ref{ModelA} or \ref{ModelB} (defined in Section \ref{sec:models}) and let $\mu$ denote the corresponding probability measure (constructed in detail in Section \ref{sec:prob2}). 
Then there exists $\theta_{A,B} >0$ such that for all $s\in (0,1)$, $p\in (1,\infty)$ with $sp < \theta_{A,B}$ and for all $\mt M \in \inte K^{qc}$ and $\mu$-almost every sequence $\{\vc y_k\}$ there exists a deformation $\vc y:\Omega \rightarrow \R^2$ such that $\vc y_k \rightarrow \vc y$ in $W^{1,1}(\Omega)$ and
\begin{align}
\label{eq:diff_incl1}
\begin{split}
\nabla \vc y &\in K \mbox{ a.e. in } \tilde\Omega, \\
\vc y & = \mt M \vc x \mbox{ on } \partial \tilde\Omega,
\end{split}
\end{align}
with $\vc y \in W^{1,\infty}(\tilde\Omega;\R^{2})\cap W^{1+s,p}(\tilde\Omega;\R^{2})$.
\end{thm}

\begin{rmk}
It would be possible to extend the result to domains $\tilde{\Omega} \subset \R^2$ which are more complicated (e.g. domains which can be written as controlled (in-)finite unions of rectangles). In order to avoid dealing with the associated issues and as the domain geometry does not constitute our main focus in this article, we restrict to the above model setting in which $\tilde{\Omega} = (0,1)^2$.
\end{rmk}

We emphasize that essentially all sequences (in terms of $\mu$) produced in our dynamics in the infinite iteration/long time limit lead to exactly stress-free solutions of the differential inclusion. Moreover, they have a certain fractality (and are in this sense self-organized and not completely random) as encoded in the higher Sobolev regularity result with $\nabla \vc y \in W^{s,p}(\tilde{\Omega}, \R^2)$ (see also the remarks below).
Numerical evaluations of the algorithms are presented in Section \ref{sec:sim}, in which we also discuss the lengths scale statistics involved in the solutions. We hope that this may eventually allow for comparisons with the measured length scale (and avalanche) distributions in experimental settings.

\subsection{Context}
Self-organized, critical systems and cellular automata have attracted substantial interest in systems undergoing phase transformations (see \cite{BTW87} and the large amount of literature building on this). Also in the mechanical literature there have been substantial endeavours towards understanding this more precisely, see for instance \cite{PTZ08,PTZ07, PTZ09,Z15, PTZT16, BUZZ16} and the references therein. In the context of martensitic phase transformations and self-similarity we highlight the early works \cite{RSS95, RSS95b, PLKK97} in which random, geometrically constrained models had been proposed and analyzed in the study of self-organized structures in martensitic phase transformations. Already in these, the emergence of self-similar, fractal microstructures was observed.

The models proposed in this article follow the line of ideas introduced in \cite{RSS95, RSS95b, PLKK97, BCH15, CH18, TIVP17}. It is our main objective to explore how simplified dynamics may lead to universal power law behaviour in nucleation processes as observed, for instance, through acoustic emission measurements. While previous models did not take into account mechanical \emph{compatibility} conditions, by connecting the probabilistic models from above with convex integration building blocks, our model does take this into account. 
In particular it allows us to link the ``self-organized'' model dynamics from \cite{BCH15, CH18, TIVP17}, convex integration schemes \cite{MS1,MSyl, MS} -- which have a natural dynamic interpretation -- and the recently obtained higher Sobolev regularity results for convex integration solutions \cite{RZZ16, RZZ18, RTZ19, DPR}. It is our hope that with further simulations, experiments and analytical investigations these connections can be strengthened and that eventually the obtained regularity exponents can be compared to the observed universal exponents of the (length scale) statistics in the experiments. From a mathematically point of view, the connection of the proposed model and ``random'', average convex integration algorithms in which only the \emph{average} instead of \emph{tailor-made packings} are considered also seems to be of independent interest (we also refer to \cite{K1} and \cite{Pompe1} for random walk interpretations of convex integration procedures). We emphasize that our model should be viewed as a \emph{hybrid} model connecting the ideas from \cite{PTZ08,PTZ07, PTZ09,Z15, PTZT16, BUZZ16} and from convex integration with higher Sobolev regularity from \cite{RZZ16, RZZ18, RTZ19, DPR}. For the sake of mathematical simplicity in this first treatment of probabilistic models involving convex integration we separate the two ingredients, the probabilistic point of view and the convex integration scheme as much as possible.
Building on this, as next steps, possibly slightly more natural algorithms could include a simultaneous iteration of the convex integration schemes and the random choice of the nucleation spots and the building block directions. These (possibly energetically more justified) models however lead to significantly more complicated analytical problems. Seeking to introduce a coupling between the ideas of convex integration (and thus of \emph{compatibility}) and the random, geometrically contrained (and thus \emph{self-organized}) structures from \cite{PTZ08,PTZ07, PTZ09,Z15, PTZT16, BUZZ16}, we here focus on the simplest possible setting, but plan to study the indicated, more complex structures in future projects.

\subsection{Regularity, self-similarity and power law length scale distributions}
\label{sec:turb}
Last but not least, we seek to heuristically connect the regularity of solutions to \eqref{eq:diff_incl1} and the power-law behaviour of statistical quantities such as length scale distributions. Precise relations between the (maximal) regularity of solutions and scaling laws are deduced in \cite{RTZ19}. On an $L^2$ based level the higher $H^s$ Sobolev regularity of the deformation gradient $\nabla \vc y$ corresponds to the finiteness of the integral
\begin{align}
\label{eq:FT}
\int\limits_{\R^2} |\vc k|^{2 + 2s}|\F \vc y ( \vc k)|^2 d \vc k.
\end{align}
In particular this implies that $\F \vc y( \vc k)$ (the Fourier transform of $\vc y$ at $\vc k$) necessarily has a decay rate that (in an average sense) is determined by the Sobolev regularity of $\vc y$. Assuming that $\F \vc y( \vc k)$ is of a power law distribution, i.e. that $|\F \vc y( \vc k)| \sim | \vc k|^{-\alpha/2}$ for $|\vc k|\geq 1$ and some $\alpha \in \R$, the finiteness condition for \eqref{eq:FT} would imply a power law behavior of the length scales involved in $\F \vc y( \vc k)$ of the order at least $\alpha  >4 +2s$. Combined with scaling laws for the associated elastic and surfac energies, one would also be able to provide upper bounds on $\alpha$ as explained in \cite{RTZ19}. In this sense, the Sobolev regularity captures the degree of self-organization in a precise sense.
Similar, Fourier based considerations (for two-point functions) as a measure of the fractality or degree of self-organization of a solution can be found in \cite{PLKK97}.

\subsection{Outline of the remainder of the article}
The remainder of the article is organized as follows: After briefly collecting our most important notation in Section \ref{sec:notation}, we present our models in Section \ref{sec:models}. In order to fix the precise setting these are complemented with the precise probabilistic set-up in Section \ref{sec:prob2}. A first convergence result for our algorithms is discussed in Section \ref{sec:conv-algor}. In Sections \ref{sec:reg_sol} and \ref{sec:prob} the higher Sobolev regularity and the $\mu$-almost everywhere convergence of the algorithms is studied. In Section \ref{sec:tail}, as our final analytic section, we explain how (for a slight variant) of Algorithm \ref{ModelA} it is possible to dispose of the non-degeneracy condition in the algorithms and to replace this by appropriate ``tail estimates''. Last but not least, we provide several illustrations of the numerical implementation of our algorithms and their statistics. We hope that these are of use in eventually comparing our results with experimental data.

\section{Notation}
\label{sec:notation}
For the convenience of the reader, we collect some of the notation which will be used in the following sections.
We first collect the central notation from the Algorithms \ref{ModelA} and \ref{ModelB} at step $k$:
\begin{itemize}
\item $\mathcal{V}_k$ -- this is the still not transformed part of the domain $(0,1)^2$ in the iteration step $k$, it consists of a finite union of open rectangles,
\item $\mathcal{C}(\mathcal{V}_k)$ -- this is the set of connected components of $\mathcal{V}_k$,
\item $\vc p_k$ -- this is the randomly chosen point in the algorithms,
\item $\mathcal{C}(\mathcal{V}_k,\vc p_k)$ -- is the connected component of $\mathcal{V}_k$ containing $\vc p_k	$,
\item $d_k$ -- this is the randomly chosen orientation in the algorithms,
\item $\vc y_k$ -- this is the current deformation,
\item $\vc z_k^j$ -- this is the replacement building block given by Theorem \ref{ThmCI},
\item $\mathcal{B}_k$, $\mathcal{B}_k^j$ -- these are the sets on which the current deformation is replaced by a deformation which is in the wells.
\end{itemize}
In our discussion of the probabilistic background we use the following notation:
\begin{itemize}
\item  $\Omega^k := \Omega \times \dots \times \Omega$ ($k$-times) -- the $k$-fold Cartesian product of a set $\Omega \subset \R^2$,
\item $\mathcal{B}(\Omega)$ -- the Borel sets on $\Omega$,
\item $\mu$, $\mu_k$, $\rho_k$ -- the measures constructed in Lemmas \ref{lem:finitekmeasures} and \ref{lem:extension},
\item $\E$, $\E_k$ -- expectations with respect to the measures $\mu$ and $\mu_k$, by construction $\mu$ is an extension of $\mu_k$, so $\E$ reduces to $\E_k$ for finite iterations of our algorithms,
\item $|A|$ -- Lebesgue measure of a Lebesgue measurable subset $ A \subset \R^n$,
\item $\mathcal{D}(D)$ -- the descendents of a set $D$, see Definition \ref{def descendants}.
\end{itemize}

\section{The models}
\label{sec:models}

As a model setting, we consider the energy wells determined by the strains $\mt F_0$ and $\mt F_0^{-1}$ from \eqref{defK_intro} and \eqref{eq:matrices_intro}.
We remark that, as shown in \cite[Sec. 5]{BJ92}, given two wells with two rank-one connections (and the physically natural condition of equal determinant), one can always reduce the problem to our case via an affine change of variables. 

\begin{figure}[t]
  \centering
  \includegraphics[width=0.45\linewidth,page=2]{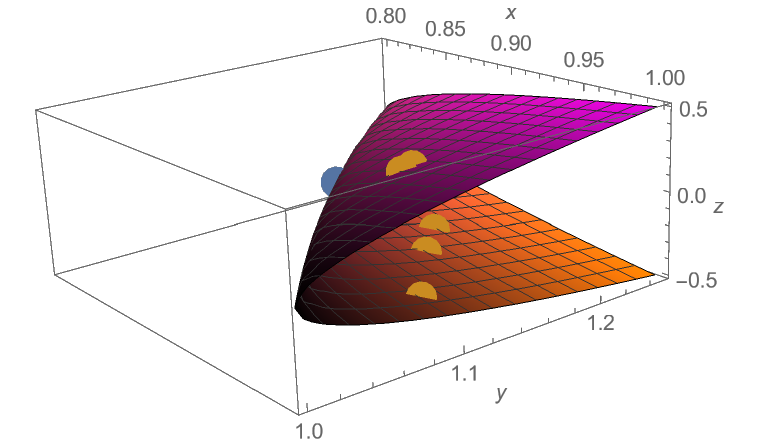}
  \includegraphics[width=0.45\linewidth,page=3]{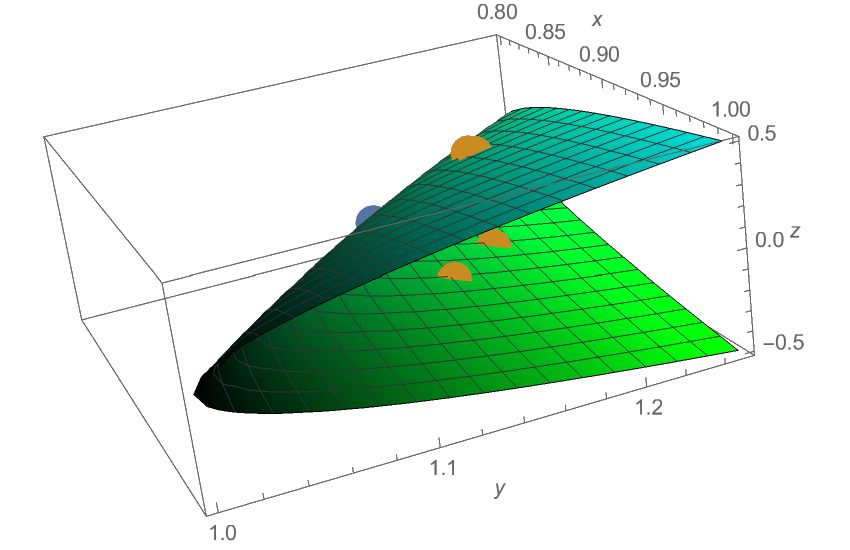}
   \caption{
	Representation of the iterative steps of the convex iteration algorithm in \cite{DPR} with boundary conditions   
  $M\approx \diag(0.939, 1.064)$, $\gamma = 0.5$
   are shown in terms of the
     Cauchy-Green tensors. Here the coordinates are $ x=(\mt M^T \mt M)_{11}\in(0,1),
     y=(\mt M^T \mt M)_{22}\in(0,1+\gamma^2)$
     for the plane directions and $z=(\mt M^T \mt M)_{12}=(\mt M^T \mt M)_{21}= \pm \sqrt{1-xy}$ for the vertical direction. The closer the steps are to $\mt F_0,\mt F_0^{-1}$ (that is to the points with $(x,y) = (1,1+\gamma^2)$) the closer the algorithm is to convergence.
     Here we use a magenta-orange-black color
     coding for the horizontal replacement (horizontal rectangle) and cyan-green-black for the
     vertical one (vertical rectangle).}
  \label{fig:vert}
\end{figure}

\begin{figure}[th]
  \centering
  \includegraphics[width=0.7\linewidth,page=4]{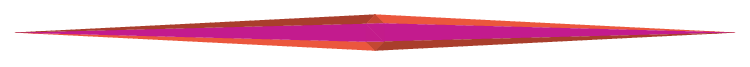}
  \hspace{0.5 cm}
  \includegraphics[height=4cm,page=5]{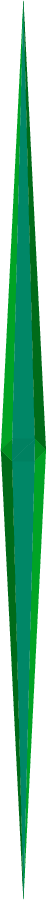}
  \caption{Given the matrix decomposition, we employ ``rhombi-construction'' on a
    diamond-shaped domain. The green colour coding of Figure \ref{fig:vert} right corresponds to vertical twins, the magenta colour coding of Figure \ref{fig:vert} left to horizontal twins. In the notation of \cite{DPR} these correspond to the coordinates $\mt F_1$ and $\mt F_2$, respectively. }
  \label{fig:diamond}
\end{figure}

Below we are going to rely on the following theorem:
\begin{thm}[Thm. 1, \cite{DPR}]
\label{ThmCI}
Let $K$ be as in \eqref{defK_intro}--\eqref{eq:matrices_intro}. Let $\tilde\Omega \subset \R^2$ satisfy 
\begin{equation}
\label{eqDomain2}
\tag{D}
\text{\parbox{3.8in}{\centering $\tilde\Omega$ is open, connected, and can be covered (up to a set of measure zero) by finitely many open disjoint triangles.}}
\end{equation}
Then there exists $\theta_0 >0$ (independent of $\tilde{\Omega}$) such that for all $s\in (0,1)$, $p\in (1,\infty)$ with $sp < \theta_0$ and for all $\mt M \in \inte K^{qc}$ there exists a deformation $\vc u:\tilde{\Omega} \rightarrow \R^2$ such that
\begin{align*}
\nabla \vc u &\in K \mbox{ a.e. in } \tilde\Omega, \\
\vc u & = \mt M \vc x \mbox{ on } \partial \tilde\Omega,\\ 
\vc u &\in W^{1,\infty}(\tilde\Omega;\R^{2})\cap W^{1+s,p}(\tilde\Omega;\R^{2}).
\end{align*}
\end{thm}

We will use the solutions from Theorem \ref{ThmCI} as building blocks for our
``plates'' (see Step (ii) of the probabilistic nucleation algorithms explained in the introduction, see also Figure \ref{fig:buildingblocks} for an
illustration of a building block for $\mt M\approx\begin{pmatrix}
    0.939 & 0\\
    0 & 1.064
  \end{pmatrix}$ and $\gamma=0.5$).

\begin{figure}[htb]
  \centering
  \includegraphics[width=0.8\linewidth]{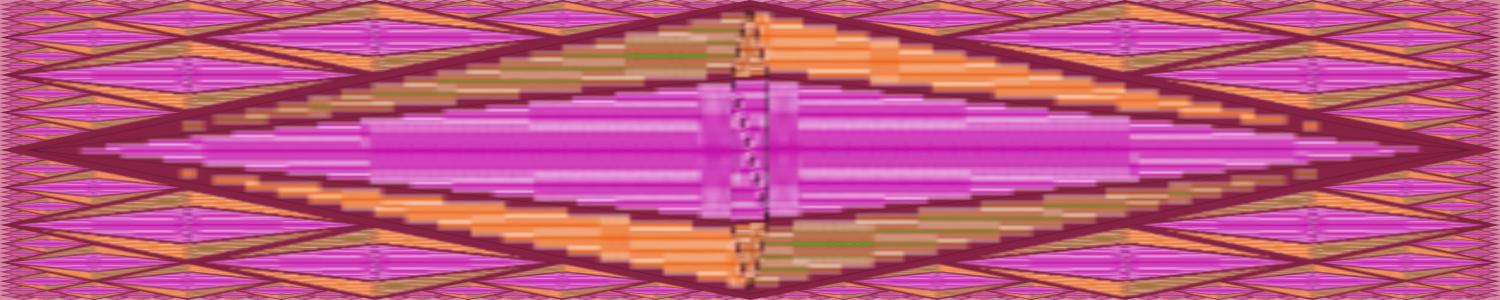}
  \caption{We use Theorem \ref{ThmCI} to construct building blocks in the shape of
    rectangles (which themselves are covered by rhombi-constructions in the form of diamonds, see Figure \ref{fig:diamond}). Different colours here correspond to different values of $\nabla
    \vc u$ with the colour coding given as in Figure \ref{fig:vert}.  
    For horizontal rectangles in the algorithms below we always begin with a decomposition along horizontal laminates, i.e. in the magenta-orange colour coding scheme.    
    In particular from the colours in this figure it is clear that the underlying deformation is not yet a full solution (but only a subsolution, roughly speaking an approximate solution) to the differential inclusion. The construction of solutions to Theorem \ref{ThmCI} is iterative. We have here depicted a subsolution obtained after three iterations.
    }
  \label{fig:buildingblocks}
\end{figure}

The solutions from Theorem \ref{ThmCI} are obtained iteratively through the method of convex integration, by iteratively deforming the current gradient distribution into an increasingly favourable one, eventually in the infinite iteration limit passing to a solution of the full differential inclusion \eqref{eq:diff_incl1} (see Figure \ref{fig:vert}).

Here in each step we cover a rectangle in the given domain by ``rhombi\hyp constructions'' (see Lemma 4.1 and Figure 4 in \cite{DPR}, building on the rhombi-constructions from the works \cite{C, CT05, MS1}) which are needle-like basic building blocks (see Figure \ref{fig:diamond}).

We remark that in this iterative replacement of deformation gradients, there are two favoured orientations for the rhombi-constructions (and thus for building blocks). These correspond to the horizontal and vertical rank-one directions which are present between the wells (see Lemma 4.1 and Figure 4 in \cite{DPR}). Thus, the choice of the orientation (of the needle-like nucleation domains) which was only heuristically justified in (ii) in \cite{BCH15, CH18, TIVP17} now becomes a rigorously justified consequence of compatibility. In order to avoid additional difficulties in the covering estimates and to keep closer to the models from \cite{BCH15,CH18,TIVP17}, we do not directly work with the diamond-shaped rhombi-constructions as the basic building blocks but consider rectangles oriented according to the rhombi-constructions which are then themselves covered by rhombi-constructions (see Figures \ref{fig:diamond} and \ref{fig:buildingblocks}).

\begin{rmk}
Instead of focusing on the geometrically non-linear two-well problem, we could also have used the results in \cite{RZZ16} or in \cite{RZZ18} instead of Theorem 1 in \cite{DPR}. As a consequence, all the results which are deduced below for the geometrically non-linear two-well problem would similarly hold in these settings. 
\end{rmk}

\begin{figure}[t]
  \begin{center}
  \includegraphics[width=0.45\linewidth,page=6]{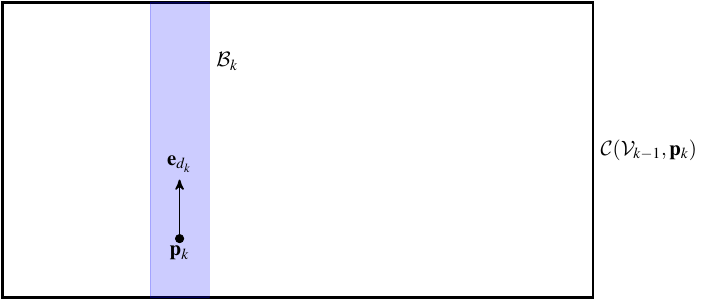}
  \includegraphics[width=0.45\linewidth,page=7]{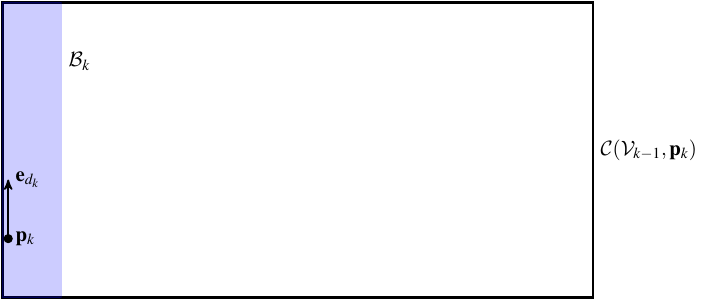}
  \includegraphics[width=0.45\linewidth,page=8]{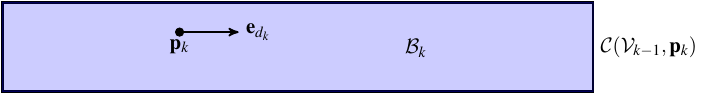}
    \end{center}
  \caption{Given a point $\vc p_k \in \mathcal{C}(\mathcal{V}_k,\vc p_k)$ and a
    direction $\vc e_{d_k}$ we insert a maximal rectangle $\mathcal{B}_k\ni \vc p_k$ of
    aspect ratio $1:\delta$.
    Here, generically the rectangle is centered around $\vc p_k$ (left). If $\vc p_k$
    is too close to the boundary, we instead shift the rectangle $\mathcal{B}_k$
    to touch the boundary (right).
    If the domain $\mathcal{C}(\mathcal{V}_k,\vc p_k)$ is too narrow (pictured
    on the bottom), we instead pick $\mathcal{B}_k=\mathcal{C}(\mathcal{V}_k,\vc p_k)$.}
  \label{fig:Bk}
\end{figure}

With this background, we next introduce two possible models for simplified, geometrically constrained, \emph{mechanically compatible} nucleation dynamics. We emphasize that these dynamics are purely phenomenological and are not derived from first principles. Their main objective is to provide further insight into the observed phenomena of universal exponents in martensitic phase transformation by means of simplified dynamics now \emph{including compatibility}. Further we seek to indicate how convex integration algorithms could naturally play a role in these types of dynamics.

\begin{alg}[Model A]
\label{ModelA}
Let $\Omega=(0,1)^2$, $\delta\in(0,1)$ and:
\begin{itemize}
\item $\mt M\in K^{qc}$;
\item $\vc y_0 := \mt M\vc x$ in $\Omega$;
\item $\mathcal{V}_0 = \Omega$.
\end{itemize}
Then, for any $k\in\mathbb{N}$
\begin{itemize}
\item let $\mathcal{C}(\mathcal{V}_{k-1})$ be the set of connected components of $\mathcal{V}_{k-1}$ (these are at most $2^{k}$ rectangles)
\item let $\vc p_k: \mathcal{C}(\mathcal{V}_{k-1}) \to\R^2$ be a function associating to each element $D_{k-1}^j\in \mathcal{C}(\mathcal{V}_{k-1})$ a point $\vc p_k^j $ chosen uniformly at random in $D_{k-1}^j$
\item let $d_k: \mathcal{C}(\mathcal{V}_{k-1}) \to \{1,2\}$ be a function associating to each $D_{k-1}^j \in \mathcal{C}(\mathcal{V}_{k-1})$ an orientation $d_k^j$ (horizontal vs vertical; modelled by the numbers 1,2) which is equal to $1$ with probability $p\in(0,1)$ and equal to $2$ with probability $1-p$. We define $(d_k^j)^{\perp}:=\{1,2\}\setminus d_k^j$
\item for each $D_k^j\in \mathcal{C}(\mathcal{V}_{k-1})$ (which is a rectangle of sides-length $\ell_1^j,\ell_2^j$) we set
\[
\mathcal{B}_k^j:=
\begin{cases}
\bigl\{\vc x \in D_k^j \colon \vc x\cdot \vc e_{d_k^j}^\perp \in (\vc p_k^j\cdot \vc e_{d_k^j}^\perp -\delta \delta_k^j\ell_{d_k^j}^j,\vc p_k^j\cdot \vc e_{d_k^j}^\perp + \delta (1-\delta_k^j)\ell_{d_k^j}^j)\bigr\},\\
\quad \quad \qquad\text{if }\delta\ell_{d_k^j}^j<\ell_{(d_k^j)^\perp}^j,\\
D_k^j, \qquad \text{if }\delta\ell_{d_k^j}^j\geq\ell_{(d_k^j)^\perp}^j,
\end{cases}
\]
where 
$$
\delta_k^j:=\argmin\left\{\left|s-\frac12\right|\colon s\in(0,1)\text{ and both }\vc p_k^j- s \ell_{d_k^j}^j\vc e_{d_k^j}^\perp,\vc p_k^j + (1-s)\ell_{d_k^j}^j \vc e_{d_k^j}^\perp \in D_k^j	\right\}
$$
\item we set $\mathcal V_{k}:= \mathcal V_{k-1}\setminus \overline{\bigcup_{j}\mathcal{B}_k^j}$ and 
\[
\vc{y}_k:=
\begin{cases}
\vc y_{k-1},\qquad &\text{on }\Omega\setminus \overline{\bigcup_{j}\mathcal{B}_k^j},\\
\vc z_{k}^j,\qquad &\text{on }\mathcal B_k^j,
\end{cases}
\]
where $\vc z_k^j\in W^{1,\infty}(\mathcal{B}_k^j;\R^2)$ is given by Theorem \ref{ThmCI}.
\end{itemize}
\end{alg}

Let us comment on this algorithm and its dynamics: We begin with a sample $\Omega$ which represents our material at the beginning of the nucleation process (e.g. with the sample being in the austenite phase or possibly also being under some prestrain).
As illustrated in Figure \ref{fig:Bk}, in each iteration step of the algorithm, in each connected component of $\mathcal{V}_k$ we randomly choose a point and an orientation, and consider a set $\mathcal{B}_k^j$ (ideally centered at the chosen point and oriented in the chosen direction, see Figure \ref{fig:Bk}) on which we replace the current deformation $\vc y_k$ by a deformation $\vc z_k^j$ which itself is given by Theorem \ref{ThmCI}. We iterate this infinitely many times, eventually obtaining a deformation which is increasingly close to being a solution to \eqref{eq:diff_incl1} (and being an exact solution in the limit $k\rightarrow \infty$). 

We remark that the main idea of the dynamics of the described algorithm is very similar to the ones proposed and analyzed in \cite{RSS95, RSS95b, PLKK97, BCH15, CH18, TIVP17}. One main difference here is that instead of just ``declaring'' the domains $\mathcal{B}_k^j$ to be filled with martensite, our domains $\mathcal{B}_k^j$ are actually filled with martensite by replacing the deformation $\vc y_{k-1}$ from the previous step by the new deformation $\vc z_k^j$ which is obtained by virtue of Theorem \ref{ThmCI}. With respect to the algorithms from \cite{RSS95, RSS95b, PLKK97, BCH15, CH18, TIVP17} by prescribing the precise deformation, our algorithm thus takes care of an \emph{additional layer of complexity} which had been ignored in the previous models.

We remark that there are several natural ways of achieving this. In our algorithm the domains $\mathcal{B}_k^j$ are immediately completely covered by a stress-free martensite configuration. As a consequence, the fully transformed sets $\mathcal{B}_k^j$ will never be modified by the algorithm again (the material is already in the energy wells). As an alternative one could, for instance, have considered an algorithm in which the diamond-shaped rhombi-constructions (see Figure \ref{fig:diamond}) are iteratively applied and which thus improve the stress distribution but do not directly yield completely stress-free configurations. In this scenario, one would then try to improve the strain distribution in the sets $\mathcal{B}_k^j$ iteratively again in later steps of the algorithm.  Mathematically the latter model would thus correspond to a ``full, random convex integration model'', while our algorithm is rather a ``hybrid, random convex integration model'', where the convex integration part is taken as a full, black-box building block as a consequence of Theorem \ref{ThmCI}. Due to the additional difficulties in combining the probabilistic perspective and the detailed convex integration estimates, we postpone the study of ``full, random convex integration algorithms'' to future work.  

In studying the length scale distribution in the sense of understanding the regularity of the final solution $\vc y$ (in expectation or $\mu$-almost everywhere), we thus need to combine an analysis of the covering algorithm (determined by the generation of the sets $\mathcal{B}_k^j$) which is essentially a probabilistic fragmentation process (and thus related to the problems in for instance \cite{FGRV95, B06} and the references therein) with the regularity of the building blocks from Theorem \ref{ThmCI}. 

We stress that in our definition of the ``nucleation sets'' $\mathcal{B}_k^j$ we allow for \emph{degenerate} sets as long as their long axis is oriented perpendicular to the long axis of the sets which are introduced through nucleation. We however exclude degenerate, too long, thin sets, if their long axis is oriented in the same direction as the sets which are inserted in the nucleation step (see the second condition in the definition of the sets $\mathcal{B}_k^j$ which is a non-degeneracy condition). 
From a technical point of view this allows us to estimate the \emph{gain in volume fraction} in each iteration step without discussing \emph{tail estimates} which originate from increasingly degenerate domains. For these the perimeter would still be controlled, the gain in the volume would however not a priori yield exponential gains in the sense of Propositions \ref{Conv:A} and \ref{Conv:B}. From a physical point of view, the degenerate vs non-degenerate choice of the rectangles $\mathcal{B}_k^j$ at this point is ad hoc. However, we believe that in more sophisticated models control on the possible degeneracies can be deduced from surface energy constraints, thus giving some credence to these type of simplifications.
As an indication in the direction of being able to derive sufficiently strong tail estimates which allow us to drop the non-degeneracy assumption, in Section \ref{sec:tail} we establish such estimates for a slightly modified algorithm. We believe that with some further effort similar results could also hold for the unmodified algorithm (see Remark \ref{rmk:comment_alg}).

Let us next discuss a second variant of our nucleation mechanism:

\begin{alg}[Model B]
\label{ModelB}
Let $\Omega=(0,1)^2$, $\delta\in(0,1)$ and:
\begin{itemize}
\item $\mt M\in K^{qc}$;
\item $\vc y_0 := \mt M\vc x$ in $\Omega$;
\item $\mathcal{V}_0 = \Omega$.
\end{itemize}
Then, for any $k\in\mathbb{N}$
\begin{itemize}
\item let $\vc p_k$ be a point chosen uniformly at random in $\mathcal{V}_{k-1}$ and we define $\mathcal{C}(\mathcal{V}_{k-1},\vc p_k)$ to be the connected component of $\mathcal{V}_{k-1}$ containing $\vc p_k$ (we remark that $\mathcal{C}(\mathcal{V}_{k-1},\vc p_k)$ is always a rectangle of size $\ell_1\times\ell_2$, with $\ell_1,\ell_2\in(0,1)$)
\item let $d_k\in \{1,2\}$ be equal to $1$ with probability $p\in(0,1)$ and be equal to $2$ with probability $1-p$. We define $d_k^{\perp}:=\{1,2\}\setminus d_k$
\item we set
\[
\mathcal{B}_k:=
\begin{cases}
\bigl\{\vc x \in \mathcal{C}(\mathcal{V}_{k-1},\vc p_k)\colon \vc x\cdot \vc e_{d_k}^\perp \in (\vc p_k\cdot \vc e_{d_k}^\perp - \delta_k\ell_{d_k},\vc p_k\cdot \vc e_{d_k}^\perp + (1-\delta_k)\ell_{d_k})\bigr\},\\
 \qquad \quad \quad \text{if }\delta\ell_{d_k}<\ell_{d_k^\perp},\\
\mathcal{C}(\mathcal{V}_{k-1},\vc p_k), \qquad \text{if }\delta\ell_{d_k}\geq\ell_{d_k^\perp},
\end{cases}
\]
where 
\begin{align*}
\delta_k
&:=\argmin\left\{\left|s-\frac12\right|\colon s\in(0,1) \right.\\
& \quad \left. \text{ and both }\vc p_k- s \ell_{d_k}\vc e_{d_k}^\perp,\vc p_k + (1-s)\ell_{d_k} \vc e_{d_k}^\perp \in \mathcal{C}(\mathcal{V}_{k-1},\vc p_k)	\right\}
\end{align*}
\item we set $\mathcal V_{k}:= \mathcal V_{k-1}\setminus \overline{\mathcal{B}}_k$ and 
\[
\vc{y}_k:=
\begin{cases}
\vc y_{k-1},\qquad &\text{on }\Omega\setminus\mathcal B_k,\\
\vc z_{k},\qquad &\text{on }\mathcal B_k,
\end{cases}
\]
where $\vc z_k\in W^{1,\infty}(\mathcal{B}_k;\R^2)$ is given by Theorem \ref{ThmCI}.
\end{itemize}
\end{alg}

In contrast to the Algorithm \ref{ModelA} this algorithm does not nucleate a new martensitic plate in each connected component of $\mathcal{V}_k$ but considers the more realistic (but mathematically slightly more involved) situation of a single nucleation event in each step. The position of the nucleation here is determined by the volume of the largest undeformed piece in the sample (see Figure \ref{fig:pk} for an illustration of the differences between the two algorithms).

In the following sections we analyse both algorithms, study their convergence properties (in expectation) and the regularity of the resulting deformations.

\begin{figure}[t]
  \includegraphics[width=0.45\linewidth,page=10]{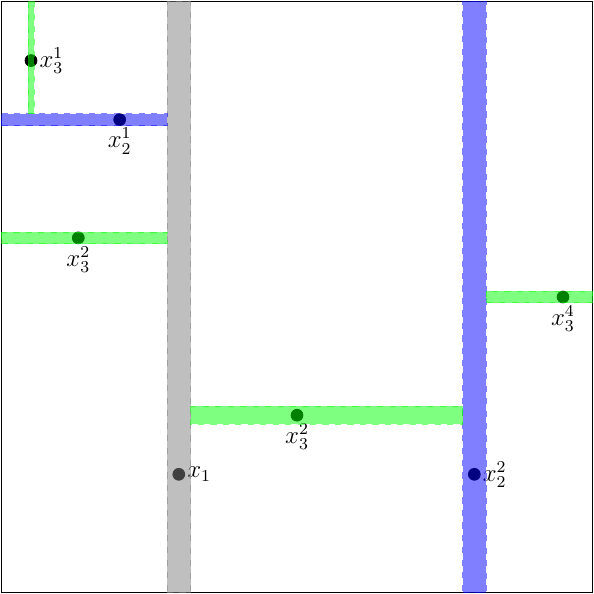}
  \includegraphics[width=0.45\linewidth,page=11]{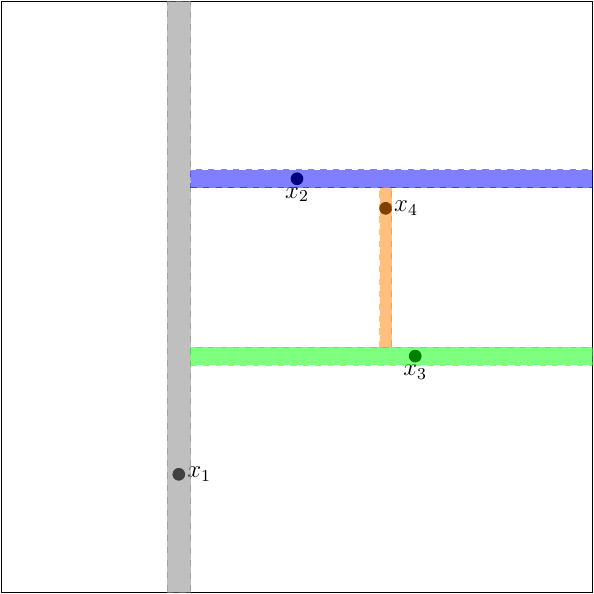}
  \centering
  \caption{In Algorithm \ref{ModelB} (pictured on the right) in each step we randomly pick
    a point in the remaining area according to the normalized Lebesgue
    measure and insert a maximal rectangle $\mathcal{B}$ containing this point.
    In Algorithm \ref{ModelA} (pictured on the left) in each step we independently pick a
    random point for each connected component. In this schematic illustration of our algorithms the colours of the rectangles correspond to the iteration step $k$ of our algorithm. In particular, we observe that in Algorithm B (right) only one set $\mathcal{B}_k$ is introduced in the step $k$ while in the Algorithm A (left) we introduce $2^k$ new sets $\mathcal{B}_k^j$ in the $k$-th iteration step. As a consequence, on average, the microstructure produced in Algorithm \ref{ModelB} provides a much more uniform covering than the one from Algorithm \ref{ModelA}, see also Figures \ref{fig:ModelA}, \ref{fig:ModelB} in Section \ref{sec:sim}.}
      \label{fig:pk}
\end{figure}

\section{Probability spaces and extensions}
\label{sec:prob2}

In the following we define the probability spaces associated to the Algorithms
\ref{ModelA} and \ref{ModelB} for each finite step $k \in \N$ and a common
probability space $(X,\mathcal{F},\mu)$ which includes all finite steps. We thus consider our convex integration algorithms as stochastic processes with $k$ being interpreted as a discrete time step.
In Section \ref{sec:conv-algor} we then study the convergence of the algorithm
in $L^p$ by computing expectations of various norms of (differences of) the
sequences obtained in the constructions. In Sections \ref{sec:reg_sol} and \ref{sec:prob} we further study
higher regularity of the solutions and in particular show that our algorithms
$\mu$-almost surely produce a $W^{1+s,p}$ regular solution of the differential inclusion. 

Our probability spaces consider the sequences of points $x_k\in \Omega$
(produced by $\vc{p}_k$) and directions $d_k\in \{1,2\}$ chosen in the algorithms.
To each such sequence we may then associate a sequence of sets
\begin{align*}
\mathcal{V}_k=\mathcal{V}_k(x_1,d_1,x_2,d_2,\dots, x_k, d_k)
\end{align*}
by constructing the rectangles $\mathcal{B}_{k}(x_1,d_1,\dots, x_j,d_j)$ as prescribed in the
algorithm.
We will show that this function is measurable and that $(\mathcal{V}_k)_{k \in
  \N}$ can therefore be considered a random variable.

\subsection{The probabilistic set-up for Algorithm \ref{ModelB}}

For simplicity of notation in the following we first discuss Algorithm
\ref{ModelB} where $k$ steps correspond to choosing $k$ points $(x_1,\dots,
x_k)\in \Omega^{k}$ and directions $(d_1,\dots, d_k) \in \{1,2\}^k$. 
\begin{lem}
  \label{lem:finitekmeasures}
  Consider the sequences of points $(x_1,x_2,\dots)$ and directions $(d_1,d_2,\dots)$
  generated by Algorithm \ref{ModelB} as a stochastic process. Then the
  corresponding (pullback probability) measure can be expressed as a density.
  More precisely, for each $k \in \N$ there exists a
  probability density
  \begin{align*}
    \rho_k: (\Omega\times\{1,2\})^k \rightarrow [0,\infty),
  \end{align*}
  such that for every Borel set $B \in \mathcal{B}(\Omega^k)$ and every
  $(d_1',\dots, d_k') \in \{1,2\}^k$ the probability that Algorithm \ref{ModelB} produces a
  sequence with $(x_1,\dots, x_k) \in B$ and $(d_1,\dots,d_k)=(d_1',\dots,
  d_k')$ is given by
  \begin{align*}
  \mu_k(B \times\{d_1',\dots, d_k'\}):=  \int_{B} \rho_k(x_1,\dots, x_k; d_1', \dots , d_k') dx_1 \dots dx_k \prod_{j=1}^k P(d_j'),
  \end{align*}
  where $P(1)=p$ and $P(2)=1-p$. 
  That is, our probability measure can be written as a density with respect to the
  Lebesgue measure and a series of Bernoulli trials determining the direction of the
  rectangles.
  Furthermore, it holds that for any $k>1$ 
  \begin{align}
    \label{eq:2}
    \begin{split}
    &\quad p \int_\Omega dx_k  \rho_{k}(x_1,\dots x_k; d_1, \dots, d_{k-1}, 1) \\
    & \quad + (1-p) \int_\Omega dx_k  \rho_{k}(x_1,\dots x_k; d_1, \dots, d_{k-1}, 2) \\
    &= \rho_{k-1}(x_1,\dots x_{k-1}; d_1, \dots, d_{k-1}),
    \end{split}
  \end{align}
  that is $\mu_{k-1}$ is given by the marginal of $\mu_k$.
\end{lem}

\begin{proof}
  The first point $x_1$ generated in Algorithm \ref{ModelB} is chosen uniformly
  at random in $\Omega$ (with respect to the Lebesgue measure) and the direction
  $d_1 \in \{1,2\}$ is chosen independently with probability $(p,1-p)$.
  Thus, in this case
  \begin{align}
    \mu_1(x_1,d_1)=
    \begin{cases}
      \frac{p}{|\Omega|}dx_1 &\text{ if } d_1=1, \\
      \frac{1-p}{|\Omega|}dx_1 &\text{ if } d_1=2.
    \end{cases}
  \end{align}
  Given a point and direction $(x_1,d_1) \in \Omega \times \{1,2\}$, in Algorithm \ref{ModelB} we obtain a rectangle
  $\mathcal{B}_{1}(x_1,d_1)$ and choose $x_2$ uniformly at random (with respect to
  the Lebesgue measure) in $\Omega\setminus \mathcal{B}_{1}(x_1,d_1)$ with probability density:
  \begin{align}
    \frac{1}{|\Omega\setminus \mathcal{B}_{1}(x_1,d_1)|} (1- 1_{\mathcal{B}_{1}(x_1,d_1)}(x_2)) dx_2,
  \end{align}
  and choose $d_2$ independently.
  Thus, given $d_1',d_2'$, we may compute
  \begin{align}
    \mu_2(x_1,x_2;d_1',d_2') = P(d_1')P(d_2')  \frac{1}{|\Omega\setminus \mathcal{B}_{1}(x_1,d_1')|} (1- 1_{\mathcal{B}_{1}(x_1,d_1')}(x_2)) dx_1 dx_2,
  \end{align}
  where we note that $\mathcal{B}_{1}(x_1,d_1')$ is prescribed in a measurable way.

  More generally, given $(x_1,\dots, x_k)$, $(d_1',\dots, d_k')$, and a set $B= B_1 \times \dots \times B_k $ the conditional
  probability for the choice of the point $x_{k+1}$ is given by the normalized
  Lebesgue measure on $\Omega$ with $k$ rectangles $(R_i(x_1,\dots, x_k,d_1',\dots,
  d_k'))_{i\in \{1,\dots, 2^k\}}$ removed and, for product sets,
  \begin{align}
    \label{eq:1}
    \begin{split}
    & \mu_{k+1}(B \times B_{k+1}\times \left(d_1',\dots,d_k', d_{k+1}'\right)) \\
    &= P(d_{k+1}') \int\limits_{B_{k+1}} \frac{1_{\Omega\setminus \cup R_i(x_1,\dots x_k, d_1',\dots, d_k')}(x_{k+1})}{|\Omega\setminus \cup R_i(x_1,\dots x_k, d_1',\dots, d_k')|} dx_{k+1}\  \mu_{k}(B \times \{d_1',\dots,d_k'\}).  
    \end{split}
  \end{align}
  In particular, as the conditional probabilities are normalized, the marginal
  property \eqref{eq:2} immediately follows.
\end{proof}

Having constructed probability spaces for each finite $k$, we now construct an
extension $((\Omega\times\{1,2\})^{\N}, \mathcal{F}, \mu)$ which includes all
these measures as restrictions.
In the case of independent measures this would correspond to identifying the above
measures with a premeasure on cylinder sets, constructing the product $\sigma$ algebra and using Caratheodory's extension theorem.
For our case we rely on the following more general extension theorem for
discrete time stochastic processes.

\begin{thm}[Theorem 3.3.6 in \cite{chung2001course}]
  \label{thm:kolmogorovextension}
  Let $m, n \in \N$, $1\leq m<n$ and define $\pi_{mn}$ to be the embedding map
  of the Borel $\sigma$ algebra $\mathcal{B}^{m}$ on $\R^m$ into $\mathcal{B}^n$ given by
  \begin{align*}
    \forall B \in \mathcal{B}^m: \pi_{mn}(B)=\{(x_1,\dots, x_{n}): (x_1,\dots, x_m) \in B\}.
  \end{align*}
  Suppose that for each $n \in \N$, $\mu_n$ is a probability measure on
  $(\mathbb{R}^n, \mathcal{B}^n)$ such that
  \begin{align}
    \label{eq:3}
    \forall m<n: \mu_n\circ \pi_{mn}= \mu_m.
  \end{align}
  Then there exists a probability space $(X, \mathcal{F}, \mu)$ and a sequence
  of random variables $X_j$ such that for each $n$, $\mu_n$ is the
  $n$-dimensional probability measure of the vector $(X_1,\dots, X_n)$.
\end{thm}

Following the argument in \cite{chung2001course}, we may
apply this extension theorem to the sequence of probability measures
generated by Algorithm \ref{ModelB}, obtaining a probability measure on the
space of sequences $(\Omega \times\{1,2\})^{\N}$.

\begin{lem}
  \label{lem:extension}
  Let $\mu_k$ be the sequence of probability measures on $(\Omega
  \times\{1,2\})^k$ as in Lemma \ref{lem:finitekmeasures} with the product Borel
  $\sigma$ algebra for each $k$.
  Let $X=(\Omega \times\{1,2\})^{\N}$ be the Cartesian product
  equipped with the product $\sigma$ algebra.
  Then there exists a measure $\mu$ on $X$ and a sequence of random variables
  $X_j$ such that $\mu_k$ is the probability measure of the
  vector $(X_1,\dots, X_k)$.
\end{lem}

\begin{proof}
  We consider probability measures on
  $\Omega\times \{1,2\}$, which can be considered as a (two-dimensional) subset of $\R^3$.
  The marginal property \eqref{eq:3} is satisfied by \eqref{eq:2} and we hence
  conclude by applying Theorem \ref{thm:kolmogorovextension}.
\end{proof}

\subsection{The probabilistic set-up for Algorithm \ref{ModelA}}

It remains to discuss Algorithm \ref{ModelA}. Here, the choice of $(x_1,d_1)$ is
identical to Algorithm \ref{ModelB}, but in the $k$-th step we choose not just
one point $x_k$ but rather $2^k$ points $(x_k^i)_{i=1}^{2^k}$, one for each connected
component.

\begin{lem}
  \label{lem:extensionA}
  Let $\mu_k$ be the sequence of probability measures on $(\Omega
  \times\{1,2\})^{2^{k}}$ generated by $k$ steps of Algorithm \ref{ModelA}.
  Then there exists a measure $\mu$ on $(\Omega \times\{1,2\})^{\N}$ and a sequence of random variables
  $X_j$ such that $\mu_k$ is the probability measure of the
  vector $(X_1,\dots, X_{2^k})$.
\end{lem}

\begin{proof}
We note that given the points obtained in step $k$, the algorithm picks all these points
independently at the same time.
In view of the extension of Theorem \ref{thm:kolmogorovextension} we further construct a
sequence of intermediate measures
\begin{align*}
\mu_{k}^{1}((x_{k-1}^{i}), x_{k}^1),\mu_{k}^{2}((x_{k-1}^{i}), x_{k}^1, x_{k}^2), \dots,
\end{align*}
where we pick them sequentially from the connected components (since these points are chosen
independently we may pick in any order).
Each such measure can be written in terms of a density expressing conditional
probabilities as in \eqref{eq:1}, where instead of all of $\Omega\setminus \cup
R_i(x_1,\dots, x_k^{i}, d_1, \dots, d_{k}^i)$, we now consider the (Lebesgue)
normalized densities on each connected component.

With this convention the measures considered in Algorithm \ref{ModelA}
correspond to the subsequence $\mu_{k}:=\mu_{k}^{2^k}$.
As the points are chosen independently according to a probability measure (which
is normalized), the sequence of measures
$\mu_k^1, \dots \mu_k^{2^k}, \mu_{k+1}^1, \dots, \mu_{k+1}^{2^{k+1}},
\mu_{k+2}^{1}, \dots$ satisfies the marginal property and hence $\mu$ can be
obtained by applying Theorem \ref{thm:kolmogorovextension}.
\end{proof}

\section{Convergence of the algorithms}
\label{sec:conv-algor}

In this section we study the convergence of the Algorithms \ref{ModelA} and
\ref{ModelB} with respect to $L^p$ norms. More precisely we show that the
expected value (with respect to the measure $\mu$ of Section \ref{sec:prob2}) of
the Lebesgue measure of the sets $\mathcal{V}_k$ tends to zero as
$k\rightarrow \infty$.
In Section \ref{sec:reg_sol} we further show that the expected value of the $BV$ norms
of the associated characteristic functions does not grow too quickly and that, as
a result, the expectations of the $W^{1+s,p}$ norms of the differences $\nabla \vc y_{k+1} - \nabla \vc y_k$
form a Cauchy sequence (in $\R$).
In Section \ref{sec:prob} we then pass from statements about expectations to
statements about sequences and in particular establish convergence and higher regularity for $\mu$-almost every 
sequence.

\subsection{Convergence of Model A}
In this section we prove the following result:
\begin{prop}
\label{Conv:A}
Consider the Algorithm \ref{ModelA} (Model A), let $\mu$ be the probability measure
constructed in Lemma \ref{lem:extensionA} and let $\mathbb{E}(\cdot)$ denote the
expectation with respect to $\mu$.
Then for each $k\geq 0$, $\mathcal{V}_k$ is a random variable with respect to
$\mu$ and it holds that 
\begin{align*}
\mathbb{E}\left(|\mathcal V_k| \right) \leq \tilde{c}_A^k |\Omega|,
\end{align*}
where $\tilde{c}_A := \max\left\{p + (1-p)(1-\delta),(1-p) + p(1-\delta)\right\}\in(0,1)$ and $p \in (0,1)$ is as in Algorithm \ref{ModelA}.
\end{prop}

We note that $\mathcal{V}_k$ only depends on $((x_1,d_1),(x_2,d_2),\dots)$ in terms of the
points and directions chosen up to step $k$. Hence, the expectation $\mathbb{E}$
may equivalently be computed in terms of the measures $\mu_k$ in which case we work with $\mathbb{E}_k$ (see Section \ref{sec:notation} for the notation).

\begin{proof}
Let $D^j_k\in \mathcal{C}(\mathcal{V}_{k-1})$. We notice that 
\[
\mathbb{E}_{k+1}\left(|D_k^j\setminus B_k^j| \Big| D_k^j \right) \leq \left.
\begin{cases}
p(1-\delta)|D_k^j| + (1-p) |D_k^j|,&\qquad\text{if } \ell_1^j>\ell_2^j	  \\
(1-p)(1-\delta)|D_k^j| + p |D_k^j|,&\qquad\text{if }\ell_1^j\leq \ell_2^j
\end{cases}  \right\}
\leq \tilde{c}_A |D_k^j|,
\]
since the new rectangle covers a fraction $\delta$ of the area if a favourable orientation is chosen by the algorithm.
Here, $\mathbb{E}(V|D_{k}^j)$ corresponds to the conditional expectation with
$(x_{1},d_1,\dots, x_k,d_k)$ prescribed (see \eqref{eq:1} for the corresponding
probability density).

Integrating this estimate with respect to $(x_{1},d_1,\dots, x_k,d_k)$ (and
$\mu_k$) we obtain the expected value inequality
\begin{align*}
\mathbb{E}_{k+1}\left(|D_k^j\setminus B_k^j| \right)\leq \tilde{c}_A \mathbb{E}_k\left(|D_k^j|\right).
\end{align*}
Thus, taking the union over $j$ and exploiting the fact that $\mathcal{V}_0 = \Omega$ and that $|\mathcal{V}_k|=\bigcup_j |D_k^j\setminus B_k^j|$ proves the claim.
\end{proof}

In particular, Proposition \ref{Conv:A} implies the following convergence result:

\begin{cor}
\label{cor:A}
Algorithm \ref{ModelA} (Model A) converges in expectation, i.e.,
$$
\lim_{k\to\infty} \mathbb{E}\left(|\mathcal{V}_k|\right) =0,\qquad
\lim_{(k,l)\to(\infty,\infty)}  \mathbb{E} \left(\|\vc y_k - \vc y_l\|_{L^\infty(\Omega)}\right)=0.
$$
\end{cor}

\begin{rmk}
We emphasize that this corollary only constitutes the very first step of our analysis of the generated sequences $\{ \vc y_k\}_{k\in \N}$. In particular, the corollary does not yet ensure the convergence of the sequence $\{\vc y_k\}_{k\in \N}$ to a solution of the differential inclusion \eqref{eq:diff_incl}.
\end{rmk}

\begin{proof}
The first statement is clear from Proposition \ref{Conv:A}. For the second statement, we just notice that (supposing without loss of generality that $k<l$) 
\begin{align*}
\E \| \vc y_k -  \vc y_l\|_{L^\infty(\Omega)} &= \E\| \vc y_k -  \vc y_l\|_{L^\infty(\mathcal{V}_{k}\setminus\mathcal{V}_{l})}\\
&\leq c\E(\|\nabla \vc y_{k}-\nabla \vc y_{l}\|_{L^\infty(\Omega)}\left| \mathcal{V}_{k}\setminus\mathcal{V}_{l}\right|^{\frac{1}{2}})\leq c\tilde{c}_A^{\frac{k}{2}}|\Omega|,
\end{align*}
for some $c>0,$ and where we have used that, since $K,K^{qc}$ are bounded, $\nabla \vc y_{j}$ is bounded in $L^\infty(\Omega)$ for each $l\geq 0$. The claim thus follows.
\end{proof}

\subsection{Convergence of Model B}
In this section we prove the following result (which does not yet ensure that $\vc y_k$ converges to a solution of \eqref{eq:diff_incl}, see Theorems \ref{RegA} and \ref{thm:almostsure} for this).

\begin{prop}
\label{Conv:B}
Consider the Algorithm \ref{ModelB} (Model B), let $\mu$ be the probability measure
constructed in Lemma \ref{lem:extension} and let $\mathbb{E}(\cdot)$ denote the
expectation with respect to $\mu$.
Then for each $k\geq 0$, $\mathcal{V}_k$ is a random variable with respect to
$\mu$ and it holds that 
\begin{align*}
\mathbb{E}\left(|\mathcal V_{2k+1}| \right) \leq \tilde{c}_B \mathbb{E}\left(|\mathcal V_{k}| \right),
\end{align*}
where $\tilde{c}_B := \tilde{c}_A + (1-\tilde{c}_A)\frac{1+e^{-\frac12}}{2}\in(0,1)$ and $\tilde{c}_A$ is as in Proposition \ref{Conv:A}.
\end{prop}
In order to work with a concise notation, we recall the concept of a descendant of a domain:

\begin{defi}[Def. 3.3 in \cite{RZZ18}]
\label{def descendants}
Let $\hat D\in \mathcal{C}(\mathcal{V}_{k})$ for some $k\geq 0.$ Then we say that $\check{D}\in \mathcal{C}(\mathcal{V}_{l})$ for some $l\geq k$ is a descendant of $\hat D$ if $\check{D}\subset\hat D.$ We denote the set of all descendants of $\hat{D}$ by $\mathcal{D}(\hat{D})$.
\end{defi} 

\begin{proof}[Proof of Proposition \ref{Conv:B}]
We know that, in the setting of Model \ref{ModelB}, $\mathcal{V}_{k}$ has at most $k+1$ connected components $D_k^j$. After $k+1$ iterations of the algorithm we thus obtain that
\[
\begin{split}
\mathbb{E}\left(|\mathcal V_{2k+1}| \Big| \mathcal V_k \right) &= \sum_{j} \mathbb{E}\left(\bigl|\mathcal{D}(D_k^j)\cap \mathcal{V}_{2k+1}\bigr| \Big| \mathcal V_k \right)\\
\leq \sum_{j} \bar p_j \mathbb{E}&\left(\bigl|\mathcal{D}(D_k^j)\cap \mathcal{V}_{2k+1}\bigr|\Big| \mathcal V_k \text{ and $\vc p_l\in D_k^j$ for some $l\in\{k+1,\dots,2k+1\}$}\right)\\
&+ \sum_{j} (1-\bar p_j) | D_k^j|,
\end{split}
\]
where $\bar p_j$ is the probability that $\vc p_l\in D_k^j$ for some $l\in\{k+1,\dots,2k+1\}$. This can be computed by noticing that 
$$
1-\bar p_j =  \prod_{l=k+1}^{2k+1} \left(1 - \frac{|D_k^j|}{|\mathcal{V}_{l-1}|} \right) \leq  \left(1 - \frac{|D_k^j|}{|\mathcal{V}_k|} \right)^{k+1},
$$
from which we obtain that
\begin{align*}
\bar{p}_j \geq 1- \left(1 - \frac{|D_k^j|}{|\mathcal{V}_k|} \right)^{k+1}=: p_j.
\end{align*}
Therefore, setting $\tilde{c}_{A}:= \min\left\{p + (1-p)(1-\delta),(1-p) + p(1-\delta)\right\}\in(0,1)$ and arguing as in the proof of Proposition \ref{Conv:A}, we deduce that 
$$
\mathbb{E}\left(\bigl|\mathcal{D}(D_k^j)\cap \mathcal{V}_{2k+1}\bigr|\Big| \mathcal V_k \text{ and $\vc p_l\in D_k^j$ for some $l\in\{k+1,\dots,2k+1\}$}\right) \leq \tilde{c}_A |D_k^j|,
$$
which allows us to
estimate
\begin{equation}
\label{estimate1}
\begin{split}
\mathbb{E}\left(|\mathcal V_{2k+1}| \Big| \mathcal V_k \right) 
&\leq \sum_{j} \left(\bar p_j\tilde{c}_A  +(1-\bar p_j) \right)| D_k^j| 
\\
&\leq \sum_{j} \left( p_j\tilde{c}_A  +(1- p_j) \right)| D_k^j| \leq |\mathcal{V}_k| + (\tilde{c}_A - 1)\sum_{j} p_j | D_k^j| 
\\
&\leq \tilde{c}_A|\mathcal{V}_k| + (1 - \tilde{c}_A)\sum_{j} \left(1 - \frac{|D_k^j|}{|\mathcal{V}_k|} \right)^{k+1} | D_k^j|.
\end{split}
\end{equation}
Let now $r_j:=\frac{|D_k^j|}{|\mathcal{V}_k|}.$ We now claim that 
\begin{equation}
\label{Claim}
\sum_{j} \left(1 - r_j \right)^{k+1} r_j \leq \hat{c}_B
\end{equation}
for some $\hat{c}_B \in (0,1).$ Indeed, let 
$$
J_1 := \left\{ j\colon r_j\geq \frac1{2(k+1)}\right\},\qquad J_2 := \left\{ j\colon r_j < \frac1{2(k+1)}\right\},
$$
and note that $J:= J_1 \cup J_2$ has $k+1$ elements.
We have
$$
\sum_{j\in J_1} \left(1 - r_j \right)^{k+1} r_j \leq \sum_{j\in J_1} \left(1 - \frac1{2(k+1)} \right)^{k+1} r_j \leq e^{-\frac 12} \sum_{j\in J_1} r_j ,
$$
and
$$
\sum_{j\in J_2} \left(1 - r_j \right)^{k+1} r_j \leq \sum_{j\in J_2} r_j .
$$
Since $\# J = k+1$,
$$
\sum_{j\in J_2} r_j \leq \frac1{2(k+1)}\sum_{j\in J_2} 1\leq \frac1{2(k+1)}\sum_{j\in J} 1\leq \frac 1 2,
$$
and since $\sum_{j\in J_1} r_j=1-\sum_{j\in J_2} r_j$, we have
\begin{align*}
\sum_{j \in J} \left(1 - r_j \right)^{k+1} r_j
&\leq \sum_{j \in J_1} \left(1 - r_j \right)^{k+1} r_j + \sum_{j \in J_2} \left(1 - r_j \right)^{k+1} r_j\\
& \leq e^{-\frac{1}{2}} \sum\limits_{j\in J_1} r_j + \sum\limits_{j\in J_2} r_j\\ 
&\leq e^{-\frac 12} + (1 - e^{-\frac 12} )\sum_{j\in J_2} r_j \leq e^{-\frac 12} + (1 - e^{-\frac 12} )\frac12,
\end{align*}
which is \eqref{Claim} with $\hat{c}_B:=\frac{1+e^{-\frac12}}2.$ Therefore, combining \eqref{estimate1} and \eqref{Claim} we deduce
$$
\mathbb{E}\left(|\mathcal V_{2k + 1}| \Big| \mathcal V_k \right) \leq | \mathcal V_k| (\tilde{c}_A + (1-\tilde{c}_A)\hat{c}_B )=:\tilde{c}_B| \mathcal V_k| .
$$
We remark that $\tilde{c}_B\in(0,1)$ is independent of $k$. Thus, taking expectation, we deduce
$$
\mathbb{E}\left(|\mathcal V_{2k + 1}| \right) \leq \tilde{c}_B\mathbb{E}\left(|\mathcal V_{k}| \right).
$$
\end{proof}

In particular, Proposition \ref{Conv:B} implies the desired convergence result.

\begin{cor}
\label{cor:B}
Algorithm \ref{ModelB} (Model B) converges in expectation. That means
$$
\lim_{k\to\infty} \mathbb{E}\left(|\mathcal{V}_k|\right) =0,\qquad \lim_{(k,l)\to(\infty,\infty)} \mathbb{E}\left(\|\vc y_k - \vc y_l\|_{L^\infty(\Omega)}\right)=0.
$$
\end{cor}
\begin{proof}
Let $k\geq 1$, and $\bar{n},\bar{k}$ be defined by
$$
\bar{n} := \sup\left\{n\in\mathbb{N}\colon 2^n-1\leq k\right\} ,\qquad \bar{k} := 2^n-1.
$$
Then, by Proposition \ref{Conv:B} we have
$$
\mathbb{E}\left(|\mathcal{V}_k|\right) \leq \mathbb{E}\left(|\mathcal{V}_{\bar{k}}|\right)\leq \tilde{c}_B^{\bar{n}}|\Omega|,
$$
which implies the first claim. The second claim follows by arguing as in the case of Model A in Corollary \ref{cor:A}.
\end{proof}

\section{Regularity of the solutions}
\label{sec:reg_sol}

After having discussed the convergence in expectation of the algorithms from Model A and Model B in the previous section, we now study their expected higher  regularity properties. 

Again we begin by considering Model A first and then pass on to Model B.

\begin{thm}
\label{RegA}
There exists $\theta_A\in(0,1)$ such that, for each $(s,p)\in(0,1)\times[1,\infty)$ satisfying $sp < \theta_A$ we have that $\vc y_k$ constructed as in Algorithm \ref{ModelA} (Model A) satisfies
$$
\mathbb{E}\left(\|\nabla\vc y_k - \nabla\vc y_{k+1}\|_{W^{s,p}(\Omega)} \right )\leq C 2^{-k\alpha} 
$$
for some $C,\alpha>0$ depending on $\theta_A,p$ and $\mt M\in K^{qc}$ only.
\end{thm}

\begin{proof}
Let us first recall that for any $f\in W^{s,p}(\Omega)$ we have
$$
\|f \|_{W^{s,p}(\Omega)}^p = \|f\|_{L^p(\Omega)}^p + \int_\Omega\int_\Omega\frac{|f(\vc x) - f(\vc y)|^p}{|\vc x-\vc y|^{2+sp}}\,\mathrm{d}\vc x\,\mathrm{d}\vc y.
$$
Let us start by assuming that $sp<\theta_0$, where $\theta_0$ is as in Theorem \ref{ThmCI}. On the one hand, since $\nabla\vc y_l(x)\in K$ for $x\in \Omega \setminus\mathcal{V}_{l}$ and any $l\geq 0$ (and as $\nabla \vc y_k$ will not be changed along the iteration on that set any more) and as both $K$ and $K^{qc}$ are compact, we have
$$
\|\nabla\vc y_k - \nabla\vc y_{k+1}\|_{L^p(\Omega)}^p\leq c |\mathcal{V}_k\setminus\mathcal{V}_{k+1}|\leq c |\mathcal{V}_k|.
$$
On the other hand, setting $\mt v_k:=\nabla\vc y_k - \nabla\vc y_{k+1}$ and using that $\nabla \vc y_k = \nabla \vc y_{k+1}$ on $(\bigcup\limits_j \mathcal{B}_k^j)^c$, we observe that
\begin{align}
\label{eq:norm_nolocal}
\begin{split}
\int_\Omega&\int_\Omega\frac{\left|\mt v_k (\vc x) - \mt v_k (\vc y)\right|^p}{|\vc x-\vc y|^{2+sp}}\,\mathrm{d}\vc x\,\mathrm{d}\vc y \\
&\leq \sum_j\int_{\mathcal{B}_k^j}\int_{\mathcal{B}_k^j}\frac{\left|\mt v_k (\vc x) - \mt v_k (\vc y)\right|^p}{|\vc x-\vc y|^{2+sp}}\,\mathrm{d}\vc x\,\mathrm{d}\vc y + 2\sum_j\int_{\mathcal{B}_k^j}\int_{(\mathcal{B}_k^j)^c}\frac{\left|\mt v_k (\vc x) - \mt v_k (\vc y)\right|^p}{|\vc x-\vc y|^{2+sp}}\,\mathrm{d}\vc x\,\mathrm{d}\vc y .
\end{split}
\end{align}
The first term in \eqref{eq:norm_nolocal} can be bounded thanks to Theorem \ref{ThmCI}. Indeed,
$$
\sum_j\int_{\mathcal{B}_k^j}\int_{\mathcal{B}_k^j}\frac{\left|\mt v_k (\vc x) - \mt v_k (\vc y)\right|^p}{|\vc x-\vc y|^{2+sp}}\,\mathrm{d}\vc x\,\mathrm{d}\vc y \leq \sum_j \left( \|\nabla \vc y_{k+1} \|_{W^{s,p}(\mathcal{B}_k^j)}^p +  \|\nabla \vc y_{k} \|_{W^{s,p}(\mathcal{B}_k^j)}^p\right).
$$
Now building on the interpolation estimate (see \cite[Corollary 3]{RZZ16})
\begin{align}
\label{eq:inter_old1}
\|u\|_{W^{s,p}(\Omega)} \leq \|u\|_{L^{\infty}(\Omega)}^{1- \frac{1}{p}}(\|u\|_{L^1(\Omega)}^{1-\tilde{\theta}_0} \|u\|_{BV(\Omega)}^{\tilde{\theta}_0})^{\frac{1}{p}},
\end{align}
with $\tilde{\theta}_0 = sp$, as well as the estimates (see \cite[Proposition 7.1]{DPR} and \cite[Lemma 7.1]{DPR}, where in the latter $|\Omega|$ has to be replaced by $\Per(\Omega)$)
\begin{align}
\label{eq:inter_old2}
\begin{split}
\|\nabla \vc u_k \|_{BV(\Omega)} &\leq C 2^k \Per(\Omega),\\
\|\nabla \vc u_k\|_{L^1(\Omega)} & \leq C c^k |\Omega|,
\end{split}
\end{align}
for some constant $c \in (0,1)$ which is independent of $\Omega$ and where $\vc u_k$ is the deformation from \cite{DPR}, we obtain that for our deformation, by combining \eqref{eq:inter_old1} and \eqref{eq:inter_old2}, we have for $sp < \theta_0$ (where $\theta_0>0$ is the regularity threshold from \cite{DPR} and Theorem \ref{ThmCI})
\begin{align}
\label{eq:apriori_bd}
\|\nabla \vc y_k\|_{W^{s,p}(\mathcal{B}_k^j)} + \|\nabla \vc y_{k+1}\|_{W^{s,p}(\mathcal{B}_k^j)}\leq C \Per(\mathcal{B}_k^j)^{s} |\mathcal{B}_k^j|^{\frac{1}{p}-s}.
\end{align}
Hence, Cauchy-Schwarz and the fact that in the $k$-th iteration step there are $2^k$ sets $\mathcal{B}_k^j$ in which $\vc y_k$ is modified, implies that
\begin{align*}
\sum_j\int_{\mathcal{B}_k^j}\int_{\mathcal{B}_k^j}\frac{\left|\mt v_k (\vc x) - \mt v_k (\vc y)\right|^p}{|\vc x-\vc y|^{2+sp}}\,\mathrm{d}\vc x\,\mathrm{d}\vc y
&\leq \sum\limits_{j} \|\nabla \vc y_k \|_{W^{s,p}(\mathcal{B}_k^j)}^p  
\leq \Per(\Omega)^{sp}\sum\limits_{j=1}^{2^k} |\mathcal{B}_k^j|^{1-sp} \\
&\leq \Per(\Omega)^{sp} 2^{k sp} |\mathcal{V}_k \setminus \mathcal{V}_{k+1}|^{1-sp}.
\end{align*}

Regarding the second term in \eqref{eq:norm_nolocal}, exploiting the boundedness of the $\mt v_k$, we have for each $j$ that
\begin{align*}
&\int_{\mathcal{B}_k^j}\int_{(\mathcal{B}_k^j)^c}\frac{\left|\mt v_k (\vc x) - \mt v_k (\vc y)\right|^p}{|\vc x-\vc y|^{2+sp}}\,\mathrm{d}\vc y\,\mathrm{d}\vc x\leq c\int_{\mathcal{B}_k^j}\int_{(\mathcal{B}_k^j)^c}\frac{1}{|\vc x-\vc y|^{2+sp}}\,\mathrm{d}\vc y\,\mathrm{d}\vc x
\\
&\leq c\int_{\mathcal{B}_k^j}\int_{\left(B\left(\vc x,\dist(\vc x,\partial\mathcal{B}_k^j)\right)\right)^c}\frac{1}{|\vc x-\vc y|^{2+sp}}\,\mathrm{d}\vc y\,\mathrm{d}\vc x
\\
&\leq c\int_{\mathcal{B}_k^j}\int_{\dist(\vc x,\partial\mathcal{B}_k^j)}^\infty\frac{1}{r^{1+sp}}\,\mathrm{d}r\,\mathrm{d}\vc x
\\
&\leq c\Per(\mathcal{B}_k^j)\min\{\ell_1^j,\ell_2^j\}^{1-sp}
\\
&\leq c\Per(\mathcal{B}_k^j)^{sp}|\mathcal{B}_k^j|^{{1-sp}}.
\end{align*}
Here the estimate in the second to last line is a consequence of the following considerations:
Splitting 
\begin{align*}
\int_{\mathcal{B}_k^j}\int_{\dist(\vc x,\partial\mathcal{B}_k^j)}^\infty\frac{1}{r^{1+sp}}\,\mathrm{d}r\,\mathrm{d}\vc x
&= \sum\limits_{i=1}^2 \int_{\Delta_{j,k}^{i}}\int_{\dist(\vc x,\partial\mathcal{B}_k^j)}^\infty\frac{1}{r^{1+sp}}\,\mathrm{d}r\,\mathrm{d}\vc x \\
& \quad + \sum\limits_{i=1}^2\int_{T_{j,k}^i}\int_{\dist(\vc x,\partial\mathcal{B}_k^j)}^\infty\frac{1}{r^{1+sp}}\,\mathrm{d}r\,\mathrm{d}\vc x,
\end{align*}
where $\Delta_{j,k}^i$ denote the triangles and $T_{j,k}^i$ the trapezoids from Figure \ref{fig:distances}. Further,
in the triangle $\D_{j,k}^i$ in Figure \ref{fig:distances} we estimate from above by
  the estimate over the rectangle $[0,\ell_2^j]\times [0,\ell_2^j]$:
\begin{align*}
\int_{\Delta_{j,k}^{i}}\int_{\dist(\vc x,\partial\mathcal{B}_k^j)}^\infty\frac{1}{r^{1+sp}}\,\mathrm{d}r\,\mathrm{d}\vc x \leq   \ell_2^j \int_0^{\ell_2^j} x_1^{-sp} dx_1= \frac{1}{1-sp} \ell_2^j (\ell_2^j)^{1-sp},
\end{align*}
where we used that the integral with respect to $r$ is given by a constant times $\text{dist}^{-sp}$.
Similarly, for the trapezoids on the bottom in Figure \ref{fig:distances} we similarly estimate by
\begin{align*}
\int_{T_{j,k}^i}\int_{\dist(\vc x,\partial\mathcal{B}_k^j)}^\infty\frac{1}{r^{1+sp}}\,\mathrm{d}r\,\mathrm{d}\vc x \leq  \ell_1^j \int_0^{\ell_2^j}x_2^{-sp} dx_2 =\frac{1}{1-sp} \ell_1^j (\ell_2^j)^{1-sp}.
\end{align*}
Thus, the integral is controlled by
\begin{align*}
  \max(\ell_1^j,\ell_2^j)\min(\ell_1^j, \ell_2^j)^{1-sp}\leq \Per(\mathcal{B}_k^j)\min(\ell_1^j, \ell_2^j)^{1-sp}.
\end{align*}
\begin{figure}[htbp]
  \centering
  \includegraphics[width=0.7\linewidth,page=12]{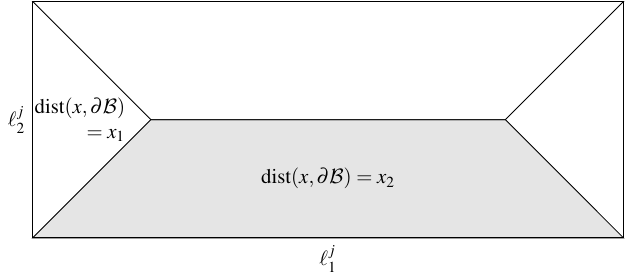}
  \caption{In order to estimate the $W^{s,p}$ seminorm over
    $\mathcal{B}_k^j\times \mathcal{B}_k^j$ we split the rectangle into the
    depicted four regions. In each of these regions the distance to the boundary is
    explicitly given in terms of the Cartesian coordinates.}
  \label{fig:distances}
\end{figure}

Thus, by H\"older's inequality and \eqref{eq:apriori_bd} we infer
\begin{align*}
&\sum_j\int_{\mathcal{B}_k^j}\int_{(\mathcal{B}_k^j)^c}\frac{\left|\mt v_k (\vc x) - \mt v_k (\vc y)\right|^p}{|\vc x-\vc y|^{2+sp}}\,\mathrm{d}\vc x\,\mathrm{d}\vc y 
\leq c\sum_j\Per(\mathcal{B}_k^j)^{sp}|\mathcal{B}_k^j|^{{1-sp}}\\
&\leq c\left(\sum_j\Per(\mathcal{B}_k^j)\right)^{sp}\left(\sum_j|\mathcal{B}_k^j|\right)^{1-sp}
\leq c\left(2^k\right)^{sp}|\mathcal{V}_k|^{1-sp}.
\end{align*}
By collecting all the above estimates we obtain for $c=c(\Per(\Omega))>0$
\begin{equation}
\label{stima big}
\|\nabla\vc y_k - \nabla\vc y_{k+1}\|_{W^{s,p}(\Omega)}^p\leq c  2^{k sp} |\mathcal{V}_k \setminus \mathcal{V}_{k+1}|^{1- sp} + c 2^{k sp} |\mathcal{V}_k|^{1-sp},
\end{equation}
which by taking the expected value and by Proposition \ref{Conv:A} becomes
\begin{align*}
\mathbb{E}\left(\|\nabla\vc y_k - \nabla\vc y_{k+1}\|_{W^{s,p}(\Omega)}^p \right) 
&\leq c \left(2^k\right)^{sp}\left(\tilde{c}_A^k |\Omega|\right)^{{1-sp}}.
\end{align*}
Therefore, choosing  $\theta_A := sp\in(0,\theta_0)$ such that $2^{\theta_A}\cdot\tilde{c}_A^{1-\theta_A} <1$, we deduce the existence of $\alpha>0$ satisfying 
$$
\mathbb{E}\left(\|\nabla\vc y_k - \nabla\vc y_{k+1}\|_{W^{s,p}(\Omega)}^p \right) 
\leq c 2^{-\alpha k}.
$$
\end{proof}

Arguing similarly as for Model A, regarding Model B we have: 

\begin{thm}
  \label{thm:ModelB}
There exists $\theta_B\in(0,1)$ such that, for each $(s,p)\in(0,1)\times[1,\infty)$ satisfying $sp < \theta_B$ we have that $\vc y_k$ constructed as in Algorithm \ref{ModelB} (Model B) satisfies
$$
\mathbb{E}\left(\|\nabla\vc y_k - \nabla\vc y_{2k+1}\|_{W^{s,p}(\Omega)} \right )\leq C 2^{-\alpha\log_2(k+1)} 
$$
for any $k\geq 1$, and for some $C,\alpha>0$ depending on $\theta_A,p$ and $\mt M\in K^{qc}$ only.
\end{thm}
We remark that for $k=2^l-1$ it holds that $2k+1 = 2^{l+1}-1$ and
  $2^{-\alpha \log_2(k+1)}=2^{-\alpha l}$. Hence, we may consider the
  subsequence $\tilde{y}_{l}:= y_{2^l-1}$ to obtain an estimate of the same form
  as in Theorem \ref{RegA}.
\begin{proof}
The proof follows the approach devised in the proof of Theorem \ref{RegA}. Again, we start by assuming that $sp<\theta_0$, where $\theta_0$ is as in Theorem \ref{ThmCI}. Then, by arguing as in the proof of Theorem
 \ref{RegA} we deduce
\begin{align*}
\|\nabla\vc y_k - \nabla\vc y_{2k+1}\|_{W^{s,p}(\Omega)}^p
&\leq c \Per(\Omega)^{sp}(k+1)^{sp} |\mathcal{V}_{k} \setminus \mathcal{V}_{2k+1}|^{1-sp} \\
& \quad + c\left(\sum_{j=k+1}^{2k+1}\Per(\mathcal{B}_j)\right)^{sp}|\mathcal{V}_k|^{1-sp}
\\
&\leq c \Per(\Omega)^{sp} (k+1)^{sp} |\mathcal{V}_{k} \setminus \mathcal{V}_{2k+1}|^{1-sp}\\
& \quad + c\left(k+1\right)^{sp}|\mathcal{V}_k|^{1-sp},
\end{align*}
where, in order to infer the estimate for the $W^{s,p}$ semi-norm, we bound the difference of $\nabla\vc y_k - \nabla\vc y_{2k+1}$ in the different $\mathcal{B}_j$, with $j\in\{k+1,\dots,2k+1\}$, rather than in the sets $\mathcal{B}_k^j.$ Thus, taking the expected value and using the estimate from Proposition \ref{Conv:B}, we arrive at
\begin{align*}
\mathbb{E}\left(\|\nabla\vc y_k - \nabla\vc y_{2k+1}\|_{W^{s,p}(\Omega)}^p \right) 
&\leq c \left((\tilde{c}_B^{\bar{n}}|\Omega|)^{1-sp} (k+1)^{sp}  + \left(k+1\right)^{sp}\left(\tilde{c}_B^{\bar{n}} |\Omega|\right)^{{1-sp}} \right)
\\
&\leq c (2^{\theta_B}\cdot\tilde{c}_B^{\frac{\bar n}n(1-\theta_B)})^n,
\end{align*}
where $n:=\log_2(k+1)$, $\bar{n}:=\left\lfloor n \right\rfloor$ and $\theta_B = sp$. Since $\frac{\bar n}n\geq\frac12$, whenever $k\geq 1$, by choosing $\theta_B:=sp\in(0,\theta_0)$ such that $2^{\theta_B}\cdot\tilde{c}_B^{\frac12(1-\theta_B)} <1$, we thus infer the claimed result.
\end{proof}

\section{Almost sure convergence and higher regularity}
\label{sec:prob}

In Sections \ref{RegA}-\ref{sec:reg_sol} we have established estimates on
solutions, their regularity and their convergence in expectation. As a consequence of these
results we further obtain convergence along sequences for $\mu$-almost every sequence produced by the algorithms.

The following theorem converts the results on expectations of Theorem \ref{RegA}
into a statement on $\mu$-almost every sequence. 

\begin{thm}
  \label{thm:almostsure}
  Consider Algorithm \ref{ModelB} and let $\mu$ be as in Lemma
  \ref{lem:extension} (or Algorithm \ref{ModelA} and $\mu$ as in Lemma
\ref{lem:extensionA}).
  Then there exists $\alpha'>0$ such that for $\mu$-almost every sequence $(x_1,d_1, \dots)$ obtained in the
  algorithm, there exists $K<\infty$ (depending on the sequence) such that for all $k\geq
  K$ it holds that
  \begin{align}
    \label{eq:4}
    \|\nabla\vc y_k - \nabla\vc y_{k+1}\|_{W^{s,p}(\Omega)} \leq \tilde{C} 2^{-\alpha' k}, 
  \end{align}
  for some constant $\tilde{C}$ independent of $k$ (in the case of Algorithm
  \ref{ModelA} we estimate $\nabla \vc y_{2^k-1} - \nabla\vc y_{2^{k+1}-1}$ instead).
  In particular, $\mu$-almost every generated sequence $\{\vc y_k\}_{k\in \N}$ is Cauchy in $W^{1+s,p}(\Omega)$ and has a limit $\vc y \in
  W^{1+s,p}(\Omega)$. The function $\vc y$ satisfies the differential inclusion problem \eqref{eq:diff_incl1}.
\end{thm}

\begin{proof}[Proof of Theorem \ref{thm:almostsure}]
  In Theorem \ref{RegA} we had shown that there exists $\alpha>0$ such that
  \begin{align*}
    \mathbb{E}(\|\nabla\vc y_k - \nabla\vc y_{k+1}\|_{W^{s,p}(\Omega)}) \leq \tilde{C} 2^{-\alpha k}.
  \end{align*}
  Let now $0<\alpha'<\alpha$ and define $C>1$ such that $2^{-\alpha' k}= C^{k}
  2^{-\alpha k}$.
  Then by Chebychev's inequality it holds that
  \begin{align*}
    \mu(\{(x_1,d_1,\dots): \|\nabla\vc y_k - \nabla\vc y_{k+1}\|_{W^{s,p}(\Omega)}\geq  2^{-\alpha'k}\}) \leq \tilde{C} C^{-k}.
  \end{align*}
  In particular, given $K\in \N$ we may define the exceptional sets
  \begin{align*}
    U_K=\bigcup_{k\geq K} \{(x_1,d_1,\dots): \|\nabla\vc y_k - \nabla\vc y_{k+1}\|_{W^{s,p}(\Omega)}\geq  2^{-\alpha'k}\}
  \end{align*}
  and by subadditivity of the measure and the geometric series we obtain that
  \begin{align*}
    \mu(U_K)\leq \tilde{C} \sum_{k\geq K} C^{-k}= \tilde{C} C^{-K} \frac{1}{1-C^{-1}}.
  \end{align*}
  Thus for every $\epsilon>0$ we may find $K\in \N$ sufficiently large such
  that $\mu(U_K)\leq \epsilon$ and therefore 
  \begin{align*}
    W_{\epsilon}:= (\Omega\times \{1,2\})^{\N}\setminus U_K
  \end{align*}
  has measure at least $1-\epsilon$ and by construction of $W_{\epsilon}$ the estimate \eqref{eq:4}
  holds for each sequence in $W_{\epsilon}$ for all $k\geq K$.

  Let now $\epsilon_j$ be some sequence with $\epsilon_j\rightarrow 0$ and define 
  \begin{align*}
    W^{\star}= \bigcup_{j} W_{\epsilon_j}.
  \end{align*}
  This set has full measure since
  \begin{align*}
    \mu((\Omega\times \{1,2\})^{\N} \setminus W^{\star})
    \leq \inf_{j} \mu((\Omega\times \{1,2\})^{\N} \setminus W_{\epsilon_j}) \leq \inf_{j} \epsilon_j =0,
  \end{align*}
  and therefore any sequence $\mu$-almost surely is in $W^{\star}$.
  By construction, for any $(x_1,d_1, \dots) \in W^{\star}$ there exists
  $\epsilon_{j}$  and hence $K$ such that \eqref{eq:4} is valid for $k\geq K$.
  
    As a consequence, $ \nabla \vc y_k \rightarrow \nabla \vc y$ in $W^{s,p}(\Omega)$ for $\mu$-almost every sequence constructed in Algorithm \ref{ModelA} and $\vc y \in W^{1+s,p}(\Omega)\cap W^{1,\infty}(\Omega)$ $\mu$-almost surely.
  
  In order to observe that $\nabla \vc y \in K$ for $\mu$-almost every sequence $(x_1,d_1,\dots)$, we argue analogously: By Proposition \ref{Conv:A} we have that $\mathbb{E}(|\mathcal{V}_k|) \leq c^k$ for some $c\in (0,1)$. Then, as above, Chebychev's inequality again yields that for some $d>1$ 
\begin{align*}
\mu(\{(x_1,d_1,x_2,d_2,\dots): \ |\mathcal{V}_k|\geq d^k c^k\})\leq d^{-k}.
\end{align*}
Given $K \in \N$, we again define exceptional sets 
\begin{align*}
\tilde{U}_K = \bigcup\limits_{k\geq K}\{(x_1,d_1,\dots): \ |\mathcal{V}_k|\geq (dc)^{k}\},
\end{align*}
and obtain that
\begin{align*}
\mu(\tilde{U}_K) \leq d^{-k} \frac{1}{1-d^{-1}}.
\end{align*}  
Defining sets $\tilde{W}_{\epsilon}$ and $\tilde{W}^{\star}$ as above, and noticing that the union of two null sets is again a null set, concludes the proof in the case of Model A. 

  The result for Algorithm \ref{ModelB} follows analogously from Theorem \ref{thm:ModelB}.
\end{proof}

\begin{proof}[Proof of Theorem \ref{thm:main}]
The proof of Theorem \ref{thm:main} is an immediate consequence of Theorem \ref{thm:almostsure}.
\end{proof}

\section{On large aspect ratios and tail estimates}
\label{sec:tail}

As remarked in Section \ref{sec:models} in our Algorithms \ref{ModelA} and
\ref{ModelB} we opted to completely cover the rectangle $D_k^j$, if
\begin{align*}
  \ell^j_{d^j_k}\geq \delta^{-1} \ell^j_{(d^j_k)^\perp},
\end{align*}
that is, if the length of the rectangle in the direction $\vc e_{d_k^j}$ we picked is too long.
This cut-off simplifies the covering arguments and allows us to more easily deduce
uniform bounds on volume fractions which are iteratively covered (as proved for instance in Propositions \ref{Conv:A} and \ref{Conv:B}).

In the following we show that for a slightly modified version of our algorithm
such a cut-off is \emph{not} required. These modifications are made precise in Algorithm \ref{ModelAmod} below and are illustrated in
Figure \ref{fig:Amod}. All other, not explicitly defined quantities are defined in the same way as in Algorithm \ref{ModelA}. 

\begin{figure}[thb]
  \centering
  \includegraphics[width=0.45\linewidth,page=13]{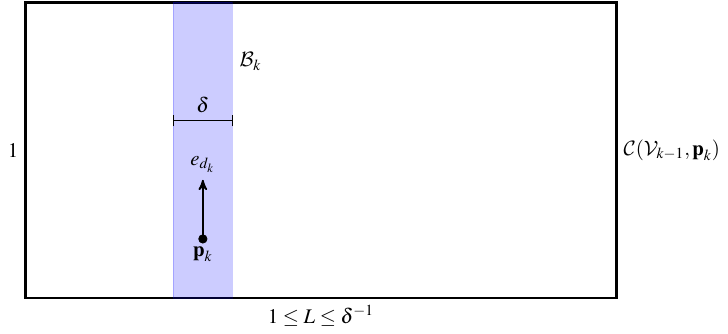}
  \includegraphics[width=0.45\linewidth,page=14]{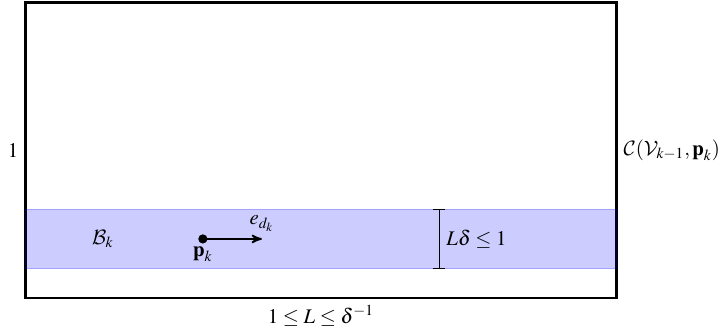}
  \includegraphics[width=0.45\linewidth,page=15]{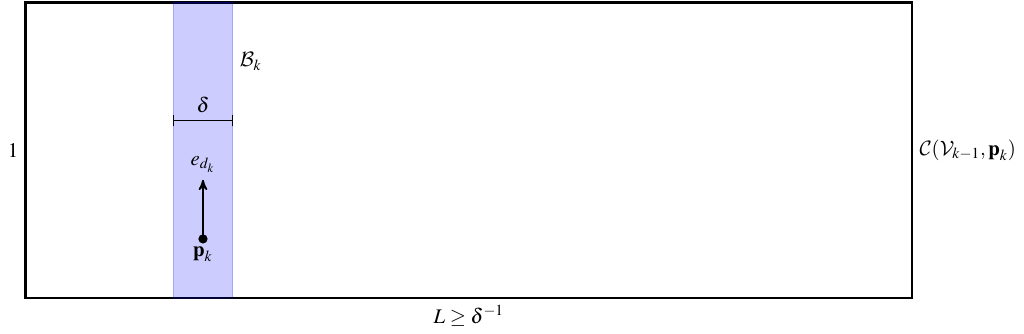}
  \includegraphics[width=0.45\linewidth,page=16]{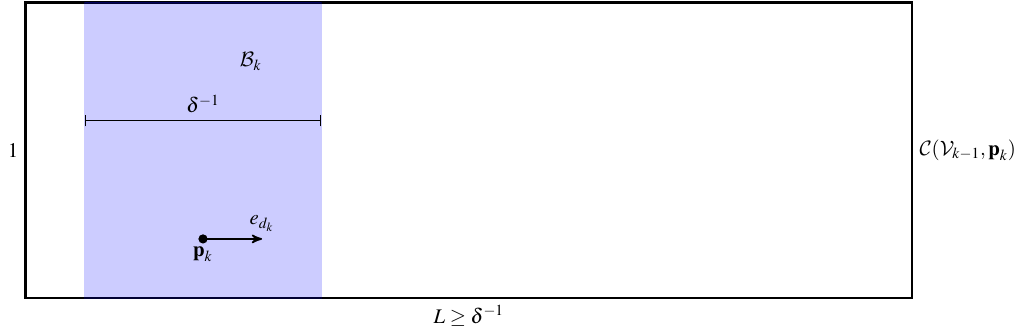}
   \includegraphics[width=0.9\linewidth,page=17]{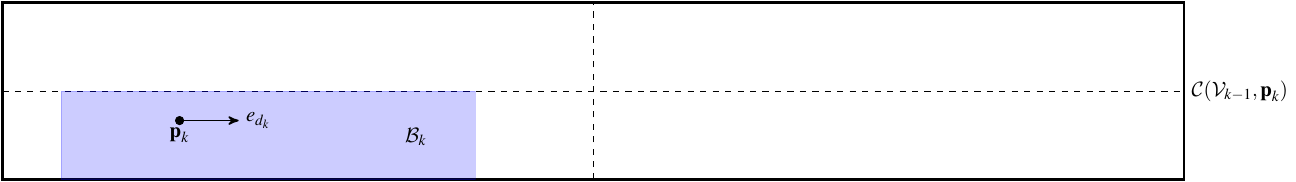}
  \caption{Top: If the rectangle $D^j_k$ has aspect ratio $1:L$ with $1\leq L\leq \delta^{-1}$ the
    inserted rectangles are translates of $(0,\delta) \times (0,1)$ and $(0,L)
    \times (0,L\delta)$.
    Center: If the aspect ratio is too long, that is $L>\delta^{-1}$, a
    rectangle $(0,L) \times (0, L\delta)$ is too tall to fit.
    We thus instead insert a translate of $(0,\delta^{-1}) \times (0,1)$.
    Bottom: If the aspect ratio satisfies $\frac{1}{2}\delta^{-1}\leq L \leq 2
    \delta^{-1}$ we only modify the quadrant $Q\subset D^{j}_k$ which contains
    the picked point $\vc p_k^j$. More precisely, we use the constructions on the
    top and center with $D^j_k$ replaced by $Q$. 
  }
    \label{fig:Amod}
\end{figure}

\begin{alg}
  \label{ModelAmod}
  We consider Algorithm \ref{ModelA} but make the following two modifications with respect to its dynamics:
  \begin{enumerate}
  \item for each $D_k^j \in \mathcal{C}(\mathcal{V}_{k-1})$ (which is a
    rectangle of side lengths $\ell_1^j, \ell_2^j$)
    we define $\mathcal{B}_{k}^j$ as follows.
    If $\ell_{d_k^j}^j<\frac{1}{2} \delta^{-1} \ell_{(d_k^j)^\perp}^j$, we keep the previous definition:
    \begin{align}
    \label{eq:replace_0}
     \mathcal{B}_k^j:= \bigl\{\vc x \in D_k^j \colon \vc x\cdot \vc e_{d_k^j}^\perp \in (\vc p_k^j\cdot \vc e_{d_k^j}^\perp -\delta \delta_k^j\ell_{d_k^j}^j,\vc p_k^j\cdot \vc e_{d_k^j}^\perp + \delta (1-\delta_k^j)\ell_{d_k^j}^j)\bigr\},
    \end{align}
    where $\delta_k^j$ is chosen such that $\mathcal{B}_{k}^j$ is contained in
    $D_k^j$:
    \begin{align*}
      \delta_k^j:=\argmin\left\{\left|s-\frac12\right|\colon s\in(0,1)\text{ and both }\vc p_k^j- s \ell_{d_k^j}^j\vc e_{d_k^j}^\perp,\vc p_k^j + (1-s)\ell_{d_k^j}^j \vc e_{d_k^j}^\perp \in D_k^j	\right\}.
    \end{align*}
    See Figure \ref{fig:Amod} (top) for an illustration.
    
    If $\delta\ell_{d_k^j}^j> 2\ell_{(d_k^j)^\perp}^j$, we cannot insert a
    translate of $(0,\ell_{d_k^j}^j) \times (0,\delta\ell_{d_k^j}^j)$, since
    $(0,\delta\ell_{d_k^j}^j) \not \subset (0,\ell_{(d_k^j)^\perp}^j)$.
    We thus instead insert a translate of $(0, \delta^{-1}
    \ell_{(d_k^j)^{\perp}}^j) \times (0,\ell_{(d_k^j)^\perp}^j)$ according to
    the following definition:
    \begin{align}
    \label{eq:replace}
    \begin{split}
      \mathcal{B}_k^j&:= \bigl\{\vc x \in D_k^j \colon \\
      & \quad \vc x \cdot \vc e_{d_k^j}  \in (\vc p_k^j\cdot \vc e_{d_k^j} - \lambda_{k}^j \delta^{-1} \ell_{(d^j_k)^{\perp}} ,\vc p_k^j\cdot \vc e_{d_k^j} + (1-\lambda_{k}^j)\delta^{-1} \ell_{(d^j_k)^{\perp}})  \bigr\},
      \end{split}
    \end{align}
    where
    \begin{align*}
      \lambda_{k}^j&:= \argmin\{|s-\frac{1}{2}|: s \in (0,1)\\
      & \quad \text{ and both } \vc p_k^j- s \delta^{-1}\ell_{(d^j_k)^{\perp}} \vc e_{d^j_k}, \vc p_k^j+ (1-s) \delta^{-1} \ell_{(d^j_k)^{\perp}} \vc e_{d^j_k} \in D_k^j \}.
    \end{align*}
    See Figure \ref{fig:Amod} (center) for an illustration.
  \item Suppose that $D_k^j \in \mathcal{C}(\mathcal{V}_{k-1})$ has lengths
    $\ell_1, \ell_2$ with $\delta^{-1}/2 \leq \frac{\ell_1}{\ell_2} \leq 2
    \delta^{-1}$ and we picked $d_{k}^j=1$.
    Then we divide $D_k^j$ into four
    quadrants and consider only the quadrant $Q$ which contains the point $\vc
    p_k^j$ which had been picked. All other three quadrants remain unchanged and
    are added to the remainder set $\mathcal{V}^{k+1}$. In the picked quadrant $Q$, we define $\mathcal{B}_k^j$ 
\begin{itemize}    
\item    by the formula \eqref{eq:replace_0} (with $D^j_k$ replaced by $Q$) if $\delta^{-1} \geq \frac{\ell_1}{\ell_2}$,
\item and by the formula \eqref{eq:replace} (with $D^j_k$ replaced by $Q$) if $\frac{\ell_1}{\ell_2} \geq  \delta^{-1}$.
\end{itemize}
    See Figure \ref{fig:Amod} (bottom) for an illustration.
  \end{enumerate}
\end{alg}
For simplicity of presentation in the following we further restrict to the case where the horizontal and vertical directions are chosen with equal probabilities $p=1-p=\frac{1}{2}$. 

\begin{rmk}
\label{rmk:2}
The first point (1) in Algorithm \ref{ModelAmod} is a relaxation of Algorithm \ref{ModelA} where we allow for rectangles with large aspect ratio also in the case of alignment with the replacement. The condition (2) is a technical assumption which we do not expect to be necessary. It ensures that we do not cover ``too much'' volume, see Lemma \ref{lem:gain} and Remark \ref{rmk:improve}. 
\end{rmk}

We show that for the modified Algorithm \ref{ModelAmod} in expectation only a (uniformly
bounded) fraction of the total volume is covered by very long rectangles while most of the volume is covered by non-degenerate rectangles.
In order to make this more precise we sort rectangles into \emph{buckets}
according to their aspect ratio.

\begin{defi}
  Let $\mathcal{V}^k$ be a collection of rectangles (defined in Algorithm \ref{ModelAmod}), let $0<\delta<1$ and
  $1<\lambda$ be given and for simplicity of notation assume that
  $\delta=\lambda^{-J+1/2}$ for some positive integer $J \in \N$.
  We then say that a rectangle $R \in \mathcal{V}^k$ is in the \emph{class} $C_{j}, j
  \in \{J, J-1, \dots, 0, -1, \dots \}\subset \Z$, if its aspect ratios $1:L$,
  $L\geq 1$ satisfies
  \begin{align*}
    \lambda^{-j-1/2}\delta^{-1}\leq L < \lambda^{-j +1/2} \delta^{-1}.
  \end{align*}
  We note that $C_l$ with $l>0$ corresponds to aspect ratios $1:L$ with $1\leq L <
  \delta^{-1}\lambda^{-1/2}$ and $C_l$ with $l<0$ corresponds to ``long'' rectangles with aspect
  ratio $L> \delta^{-1}\lambda^{1/2}$.

  Furthermore, we introduce the corresponding volumes
  \begin{align*}
    V_{j}^k= \sum_{R \in \mathcal{V}^k: R \in C_j}|R|.
  \end{align*}
  We note that the the total volume $|\mathcal{V}^k|$ satisfies
  \begin{align*}
    |\mathcal{V}^k|= \sum_{-\infty \leq j \leq J} V_{j}^k=\sum_{R \in \mathcal{V}^k} |R|. 
  \end{align*}
\end{defi}

Our objective in the following is to show that if $\mathcal{V}^{k}$ is the
random variable given by our Algorithm \ref{ModelAmod}, then there exists $J_1<0$ such
that for all $k$ it holds that
\begin{align}
  \label{eq:100}
  \sum_{j\leq J_1} \mathbb{E}(V_{j}^k)\leq 0.1 \ \mathbb{E}(|\mathcal{V}^k|).
\end{align}
We call this a \emph{tail estimate} since it shows that the contribution of
$j\leq J_1$ (which corresponds to long, thin rectangles) to the sum of the sequence $(\mathbb{E}(V_j^k))_{j\leq J}$ is small.

Supposing for the moment that this estimate holds, we deduce that the
expectation of the remaining volume decreases at an exponential rate.

\begin{thm}
  \label{thm:exp}
  Suppose that for some $J_1<0$ the random process generated
  by Algorithm \ref{ModelAmod} satisfies the estimate \eqref{eq:100} for all $k
  \in \N$. Let $\delta \in (0,0.1)$ be the threshold from Algorithms \ref{ModelA} and \ref{ModelAmod}.
  Then there exists $c=c(J_1,\delta) \in (0,1)$ such that for all $k$
  \begin{align*}
    \mathbb{E}(|\mathcal{V}^{k+1}|)\leq c \ \mathbb{E}(|\mathcal{V}^{k}|)
  \end{align*}
  and, as a consequence,
  \begin{align*}
    \mathbb{E}(|\mathcal{V}^{k}|)\leq c^{k} |\Omega|.
  \end{align*}
\end{thm}
We thus obtain similar results as in Section \ref{Conv:A} even without
completely covering long rectangles, however possibly with worse rates.

\begin{proof}[Proof of Theorem \ref{thm:exp}]
  We recall that in  Algorithm \ref{ModelAmod} we \emph{independently} insert a building block into each
  rectangle. For any given rectangle $R$ ($=D_k^j \in \mathcal{V}^k$) we may thus compute the
  expected volume fraction (of $R$) covered by the inserted building block and
  subsequently sum over all rectangles.
  We note that, by scaling, this volume fraction only depends on the aspect ratio of $R$ which
  is comparable to $\lambda^{-l} \delta^{-1}$ if $R \in C_l$.

  We claim that there exists a sequence of coefficients $c_j\in (0,1)$ (which is independent of $k$) such that
  \begin{align}
    \label{eq:200}
    \mathbb{E}(|\mathcal{V}^k|)-\mathbb{E}(|\mathcal{V}^{k+1}|) \geq \sum_{J_1<j \leq J} c_j \mathbb{E}(V_j^k).
  \end{align}
  We remark that the quantity on the left equals the total volume covered by
  building blocks when passing from step $k$ to $k+1$ (since only building blocks
  are removed). In the following we will thus have to estimate the expected volume
  fraction covered by building blocks for any given rectangle $R \in \mathcal{V}^k$. Before proving \eqref{eq:200} let us
  discuss how it allows us to conclude our proof.
To this end,  we may further estimate the right-hand-side of \eqref{eq:200} by invoking \eqref{eq:100}
  \begin{align*}
    \min(c_j) \sum_{J_1< j \leq J} \mathbb{E}(V_j^k) \geq 0.9 \min(c_j) \mathbb{E}(|\mathcal{V}^k|).
  \end{align*}
  Inserting this estimate back into \eqref{eq:200} we deduce that
  \begin{align*}
    \mathbb{E}(|\mathcal{V}^{k+1}|) \leq (1-0.9\min(c_j)) \mathbb{E}(|\mathcal{V}^k|),
  \end{align*}
  which yields the result of Theorem \ref{thm:exp}.

  It hence remains to prove the claimed inequality \eqref{eq:200}.
  Let thus $R$ be a given rectangle of lengths $\ell_1, \ell_2$ and for
  simplicity of notation denote
  $L=\frac{\max(\ell_1,\ell_2)}{\min(\ell_1,\ell_2)}\geq 1$.
  Then after rescaling, translation and possibly rotating by $\frac{\pi}{2}$ we
  may assume that
  \begin{align*}
  R=(0,L)\times (0,1).
  \end{align*}
  Let $j \in \{\dots, -1, 0, 1, \dots, J\}$ such that $R \in C_j$ and hence
  \begin{align}
    \label{eq:8}
    \lambda^{-j-1/2} \delta^{-1} \leq L \leq \lambda^{-j+1/2} \delta^{-1}.
  \end{align}
  We then estimate the expected volume covered by the inserted building block as
  follows:

  If the direction $\vc e_{d^j_k}$ picked is vertical, $\vc e_{d^j_k}= \vc e_2$, and $L \not \in
  (\frac{1}{2}\delta^{-1}, 2 \delta^{-1})$ then we
  insert a translate of $(0,\delta) \times (0,1)$ into a rectangle
  $(0,L)\times(0,1)$ (see Figure \ref{fig:Amod} on the left side of the top and
  center rows).
  If instead $L \in (\frac{1}{2}\delta^{-1}, 2 \delta^{-1})$, then by point
  $(2)$ of Algorithm \ref{ModelAmod} we only modify a quadrant and hence
  insert a translate of $(0,\delta/2) \times (0,1/2)$ (see Figure \ref{fig:Amod}
  bottom).
  In both cases we cover at most a volume fraction
  \begin{align}
    \label{eq:9}
    \frac{\delta}{L}\leq \lambda^{j-1/2}\delta^{2}
  \end{align}
  and recall that we picked the direction $e_{d^j_k}=e_2$ with a probability $\frac{1}{2}$.
  
  Suppose the direction picked is horizontal, that is $\vc e_{d^j_k}= \vc e_1$, and let again
  without loss of generality $R=(0,L) \times (0,1) \in C_j$. Then we distinguish
  three cases:
  \begin{itemize}
  \item  If $L\leq \frac{1}{2} \delta^{-1}$ (which implies that $j\geq 0$ in
    \eqref{eq:8}), we insert a translate of $(0,L)\times (0,L \delta)$ and hence
    cover a volume fraction $L\delta \leq \lambda^{-j-1/2}= \lambda^{-|j|-1/2}$. 
  \item If $L\geq 2 \delta^{-1}$ (which implies that $j\leq 0$ in
    \eqref{eq:8}), we insert a translate of $(0,\delta^{-1}) \times(0,1)$ and hence
    cover a volume fraction $\frac{\delta^{-1}}{L}\leq \lambda^{j-1/2}= \lambda^{-|j|-1/2}$.
  \item Finally, if $\frac{1}{2} \delta^{-1} \leq L \leq 2 \delta^{-1}$ (which
    implies that $|j|$ is small in \eqref{eq:8}), we only
    modify $R$ inside a quadrant and hence may bound the volume fraction
    covered from above by $\frac{1}{4}$.
  \end{itemize}
  We thus cover at most a volume fraction $\max(\frac{1}{4},
  \lambda^{-|j|-1/2})$ and recall that we picked $\vc e_{d^j_k}= \vc e_1$ with
  probability $\frac{1}{2}$.
  Combining this estimate and \eqref{eq:9}, we may thus choose 
  \begin{align}
    \label{eq:6}
    c_j=\frac{1}{2} \delta^2 \lambda^{j-1/2} + \frac{1}{2} \max(\frac{1}{4},\lambda^{-|j|-1/2}) \in (0,1).
  \end{align}
  This establishes the claimed inequality \eqref{eq:200} and hence concludes the proof.
\end{proof}

\begin{rmk}
\label{rmk:optimal}
We remark that our computations of $\mathbb{E}(|\mathcal{V}^k|)-
\mathbb{E}(|\mathcal{V}^{k+1}|)$ in \eqref{eq:200} are close to being sharp.
More precisely, given the aspect ratio of a rectangle $R$ we can precisely
compute the expected volume fraction (of $R$) which is covered by the building
block. Since in our buckets $C_j$ we group ratios which differ by at most a factor
$\lambda^{\pm 1}$, we may bound these volume fractions from above and below by
constants $c_j^*$ and $c_j$ which differ from each other by a factor at most
$\lambda^{\pm 1}$ (see Lemma \ref{lem:lowerbound} for a calculation of lower bounds).
\end{rmk}

The remainder of this section is concerned with establishing the claimed
estimate \eqref{eq:100}.
More precisely, we make the stronger claim that there exists a constant $C>1$
(for our purpose the constant can, for instance, be chosen to be $C=100$) such that
for all $k\geq 0$ and all $j\leq J_1$ it holds that
\begin{align}
  \label{eq:300}
  \mathbb{E}(V^k_{j})\leq C \lambda^{j} \mathbb{E}(|\mathcal{V}^k|).
\end{align}
That is, rectangles with a very large aspect ratio comparable to
$\lambda^{|j|}\delta^{-1}$ cover an exponentially decreasing amount of the total
volume $\mathbb{E}(|\mathcal{V}^k|)$.
As $j\leq J_1 <0$ is negative, we may relate this to the geometric series in
$\frac{1}{\lambda}<1$ (starting at $|J_1|$) and after
possibly choosing $J_1$ even more negative it holds that
\begin{align*}
  C \sum_{j\leq J_1} \lambda^{j} = C \frac{\lambda^{J_1}}{1-\lambda^{-1}} \leq 0.1,
\end{align*}
which implies the desired result \eqref{eq:100}.
In order to prove \eqref{eq:300} we proceed by induction using an upper and a
lower bound given by the following two lemmas.

\begin{lem}
  \label{lem:gain}
   Let $\mathcal{V}^k$ be as above.  Suppose that \eqref{eq:300} holds for a given $k$ and $C$ large ($C=100$)
  and $\lambda=1.1$.
  Then for all $j\leq J_1<0$ it holds that
  \begin{align}
    \label{eq:10}
    \mathbb{E}(V^{k+1}_{j})\leq 0.7\ C \lambda^{j} \mathbb{E}(|\mathcal{V}^k|).
  \end{align}
\end{lem}

\begin{lem}
  \label{lem:lowerbound}
  Let $\mathcal{V}^k$ be as above. 
  Then it holds that
  \begin{align}
    \label{eq:400}
    \mathbb{E}(|\mathcal{V}^{k+1}|)\geq 0.7 \ \mathbb{E}(|\mathcal{V}^k|).
  \end{align}
\end{lem}

\begin{rmk}
\label{rmk:comment_alg}
We remark that a failure of the lower bound \eqref{eq:400} corresponds to covering
a large volume fraction in a single iteration step of the algorithm, which at first sight seems very desirable. However,
by covering this large volume fraction we might possibly loose control of relative volume
fractions (e.g. it might be that the tail is not anymore relatively small).
We believe that \eqref{eq:400} remains true also for Algorithm \ref{ModelA}, but
our current method of proof for that case only allows to derive a suboptimal
lower bound by $0.6 \ \mathbb{E}(|\mathcal{V}^k|)$, which is not sufficient to close the
argument.
For simplicity of presentation we hence opted to modify the algorithm to cover a
lower fraction in the ``best case'' (leading to the condition (2) in Algorithm \ref{ModelAmod}).
We comment on some partial results for the unmodified case at the end of this section in Remark \ref{rmk:improve}.

We emphasize that in contrast to Lemma \ref{lem:lowerbound} the result of Lemma \ref{lem:gain} is valid for both Algorithms \ref{ModelA} and
\ref{ModelAmod} and, in particular, does not need the modifications from Algorithm \ref{ModelAmod}. 
\end{rmk}
The combination of Lemmas \ref{lem:gain} and \ref{lem:lowerbound} allows us to prove \eqref{eq:300}.

\begin{proof}[Proof of the claim \eqref{eq:300} using Lemmas \ref{lem:gain} and \ref{lem:lowerbound}]
We note that initially, that is for $k=0$, $V^k_j=0$ for all $j\leq J_1$ and
thus \eqref{eq:300} is trivially satisfied. We then aim to proceed by induction.
Suppose that \eqref{eq:300} holds for a given $k$ and with $\lambda=1.1$.
  Then by Lemma \ref{lem:gain} and Lemma \ref{lem:lowerbound} it holds that
  \begin{align*}
    \mathbb{E}(V^{k+1}_{j})\stackrel{\eqref{eq:10}}{\leq} C 0.7 \lambda^{j} \mathbb{E}(|\mathcal{V}^k|) \stackrel{\eqref{eq:400}}{\leq} C \lambda^{j} \frac{0.7}{0.7} \mathbb{E}(|\mathcal{V}^{k+1}|) = C \lambda^j \mathbb{E}(|\mathcal{V}^{k+1}|),
  \end{align*}
  and the estimate \eqref{eq:300} therefore also holds for $k+1$.
  We thus conclude by induction.
\end{proof}

It remains to prove Lemma \ref{lem:gain} and Lemma \ref{lem:lowerbound}.

\begin{proof}[Proof of Lemma \ref{lem:gain}]
  We argue similarly as in equation \eqref{eq:200} in the proof of Theorem
  \ref{thm:exp} and estimate $\mathbb{E}(V^{k+1}_j)$ in terms of
  $\mathbb{E}(V^{k}_{l})$, $l \in \Z \cap \{m \leq J\}$.
  Here we use that every rectangle in $\mathcal{V}^{k+1}$ is obtained as one of the
  connected components of a rectangle $R \in \mathcal{V}^{k}$ generated by
  inserting a building block (see Figure \ref{fig:Amod}).
  More precisely, let $j \leq J_1$ be arbitrary but fixed and let $R \in
  \mathcal{V}^k$ be a given rectangle. As in the proof of Theorem \ref{thm:exp}
  after rescaling and rotation we may assume that
  \begin{align*}
    R = (0,L) \times (0,1), L\geq 1.
  \end{align*}
  Then given the random point $\vc p \in R$ and direction $\vc e_d \in \{\vc e_1, \vc e_2\}$ we
  insert a building block $\mathcal{B}=\mathcal{B}(\vc p, \vc e_d)\subset R$, which
  divides
  \begin{align*}
   R \setminus \mathcal{B} =: R_1 \cup R_2, 
  \end{align*}
  into two connected components $R_1= R_1(\vc p, \vc e_d)$
  and $R_2= R_2(\vc p, \vc e_d)$ (if $\mathcal{B}$ touches the boundary of $R$ some of these
  components might be trivial).  
  We then compute the contribution of $R \in \mathcal{V}^k$ to $\mathbb{E}(V_j^{k+1})$ by determining for which $\vc p$ and $\vc e_d$
  it holds that $R_1 \in C_j$ or $R_2 \in C_j$ (and integrating $|R_1|$ and
  $|R_2|$ with respect to the probability density) and finally sum over all $R$.

  More precisely, we claim that for any $j\leq J_1$ it holds that
  \begin{align}
    \label{eq:5}
    \begin{split}
    \mathbb{E}(V_j^{k+1})&\leq (0.5\lambda^{j+1/2} \delta + 0.5 \lambda^{j+1/2}) \mathbb{E}(V_{\geq 0}^k) \\
                         & \quad + \delta^{2} \lambda^{j+1/2} \mathbb{E}(V_{0\geq l > j}^k) \\
                         & \quad + (1-\lambda^{-2}) \mathbb{E}(V_{j}^k) \\
                         &\quad +  \sum_{l<j} \lambda^{2l-2j} (1-\lambda^{-2}) \mathbb{E}(V_l^k).
  \end{split}
  \end{align}
  Here we used the short-hand notation
  \begin{align*}
  \mathbb{E}(V_{\geq 0}^k)  := \sum\limits_{l \geq 0} \mathbb{E}(V_{l}^k), \ 
  \mathbb{E}(V_{0\geq l > j}^k):=\sum\limits_{0\geq l >j} \mathbb{E}(V_{l}^k).
  \end{align*}
  Using \eqref{eq:300} and the fact that $C$ is large, we will argue that the main
  contribution on the right-hand-side of \eqref{eq:5} is given by the last two terms.
  More precisely, inserting the estimate \eqref{eq:300}, the last two contributions
  are controlled by
  \begin{align*}
    & \quad (1-\lambda^{-2})\sum_{l \geq j} C \lambda^{2j-3l}\mathbb{E}(|\mathcal{V}^k|) 
    = C \lambda^{-j} \frac{1-\lambda^{-2}}{1-\lambda^{-3}} \mathbb{E}(|\mathcal{V}^k|).
  \end{align*}
  In particular, we observe that
  \begin{align*}
    \beta(\lambda):= \frac{1-\lambda^{-2}}{1-\lambda^{-3}}= \frac{\lambda^{-1}+1}{\lambda^{-2}+\lambda^{-1}+1}
  \end{align*}
  approaches $\frac{2}{3}$ as $\lambda$ approaches $1$.
  Inserting these estimates into \eqref{eq:5} and choosing $\lambda=1.1$, we may thus deduce that
  \begin{align*}
   \mathbb{E}(V_j^{k+1}) \leq \lambda^{j} (\delta \lambda^{\frac{1}{2}}+\delta^2 \lambda^{\frac{1}{2}}+0.5 \lambda^{\frac{1}{2}}+ 0.68C) \mathbb{E}(|\mathcal{V}^k|) < 0.7 C \lambda^j \ \mathbb{E}(|\mathcal{V}^k|),
  \end{align*}
  provided $C$ is sufficiently large compared to $0.5$.

  It thus remains to prove the estimate \eqref{eq:5}.
  Let thus $j \leq J_1$ be arbitrary but fixed. As discussed above, for
  any $R \in \mathcal{V}^k$ we determine with which probability $R_1$ and $R_2$
  are in $C_j$ by estimating the probability of the associated sets of $(\vc p, \vc e_d)$. Using this, we compute the expectations of $|R_1| 1_{R_{1} \in C_j} + |R_2|
  1_{R_2 \in C_j}$ and then compare these to the volume $|R|$ of $R$.
  
  We remark that if the aspect ratio $1:L$, $L\geq 1$ of $R$ satisfies
  $\frac{1}{2}\delta^{-1}\leq L \leq 2 \delta^{-1}$ and we are thus in case
  $(2)$ of Algorithm \ref{ModelAmod}, then the rectangles $R_1, R_2$ are further
  rescaled by a factor $\frac{1}{2}$ and hence cover $\frac{1}{4}$ of the volume
  which they would else have occupied without this modification.
  Since we only require upper bounds on $|R_1| 1_{R_{1} \in C_j} + |R_2|
  1_{R_2\in C_j}$, this gain of a factor $\frac{1}{4}$ only improves the estimates. Thus, for
  simplicity of notation in the following we establish the stronger
  estimate for the algorithm without this second modification.  
  \medskip 

  In the following let always $R \in \mathcal{V}^k$ and without loss of
  generality, after rescaling, rotating and translating let 
  \begin{align*}
    R=(0,L) \times (0,1)
  \end{align*}
  with $L\geq 1$.
  
  \underline{The contribution by $V_{\geq 0}^k$:} Suppose that $R$ is such that
  $1\leq L\leq \delta^{-1}$ (and hence $R \in C_l$ for some $l \geq 0$).
  We then want to estimate the volume of the generated rectangles $R_1, R_2$ if
  they are in $C_j$.
  Here we say that $R_1$ (or $R_2$) is \emph{vertical} if it is a translate of
  $(0,a)\times (0,1)$ for some $a \in (0,1)$ (that is the $\vc e_2$ direction is the
  longest) and otherwise call it \emph{horizontal}.

  Let us first consider the case when $R_1$ (or $R_2$) is in $C_j$ and vertical.
  Then $R_1$ is a translate of $(0,\lambda^{j} \delta \gamma)\times (0,1)$ (see
  Figure \ref{fig:Amod} left) with $\gamma \in (\lambda^{-1/2},
  \lambda^{+1/2})$ (since the class $C_j$ was defined in this way).
  We may thus roughly bound its volume
  fraction by
  \begin{align*}
    \frac{\lambda^{j+1/2} \delta}{L} \leq \lambda^{j+1/2} \delta.
  \end{align*}

  Next suppose that $R_1$ (or $R_2$) is in $C_j$ and horizontal.
  Then (by the definition of $V^k_{\geq 0}$ and $C_j$ as well as the replacements explained in Algorithm \ref{ModelAmod}) $R_1$ is a translate of $(0,L)\times (0, \alpha)$ (see Figure
  \ref{fig:Amod} top right) with $\frac{\alpha}{L} \in (\lambda^{j-1/2}\delta^{-1},
  \lambda^{j+1/2}\delta^{-1})$ and thus in particular covers a volume fraction
  less than
  \begin{align*}
   \alpha \leq \lambda^{j+1/2}. 
  \end{align*}

  \underline{The contribution by $V_{l}^k$ with $0\geq l>j$:}
  Let again $R \in \mathcal{V}^k$ and suppose that $\delta^{-1}\leq L < \delta^{-1}\lambda^{-j-1/2}$ (and thus
  $R \in C_{l}$, $0\geq l> j$).
  Since $L\geq \delta^{-1}$, the generated rectangles  $R_1, R_2$ have height $1$
  and are thus translates of $(0,\alpha) \times (0,1)$ for some $\alpha \in
  (0,L)$ (see Figure \ref{fig:Amod} center).
  As $\alpha<L < \delta^{-1}\lambda^{-j-1/2}$, it is not possible for $R_1$ or
  $R_2$ to be horizontal rectangles in $C_j$ (since $\alpha$ is too small for the rectangles to be in $C_j$).
  If $R_1$ (or $R_2$) is a vertical rectangle and in $C_j$, it is a translate of
  $(0,\lambda^{j} \delta \gamma)\times (0,1)$ with $\gamma \in (\lambda^{-1/2},
  \lambda^{1/2})$ and hence covers a volume fraction at most
  \begin{align*}
    \frac{\lambda^{j}\delta \gamma}{L} \leq \lambda^{j+1/2} \delta^2.
  \end{align*}

  \underline{The contribution by $V_{l}^{k}$ with $l\leq j$:}
  Finally, let again $R \in \mathcal{V}^k$ and suppose that $R \in C_l$ and thus
  \begin{align*}
    \lambda^{-l-1/2}\delta^{-1}\leq L <  \lambda^{-l+1/2}\delta^{-1}.
  \end{align*}
  If the generated rectangle $R_1$ (or $R_2$) is vertical, it will only cover a
  volume fraction
  \begin{align}
    \frac{\delta \lambda^{j} \gamma}{L}\leq \delta^2 \lambda^{1+j+l},
  \end{align}
  which is negibible.
  
  In the following we thus focus on estimating the expected volume fraction covered by $R_1$ in
  $C_j$ being a horizontal rectangle (which by symmetry is the same volume
  fraction as covered by $R_2$).
  Here for concreteness we again fix
  \begin{align*}
    R=(0,L) \times (0,1)
  \end{align*}
  and let $R_1=R_1(\vc p, \vc e_d)$ be the rectangle generated on the left of the building block $\mathcal{B}$ which had been removed from $R$ (see Figure
  \ref{fig:Amod}).

  \underline{The case $l=j$:}
  By construction the building block $\mathcal{B}$ inserted depends on $\vc p$ only
  in terms of its $\vc e_1$ component $ p_1$.
  We thus ask for which $p_1$ (for given $\vc e_d$) it holds that $R_1 \in C_j$ and
  require an estimate of $|R_1(p_1, \vc e_d)|$ in that case.
  Since $p_1$ was chosen according to the Lebesgue measure, we then can compute
  the volume fraction for a given $\vc e_{d'}$ by
  \begin{align*}
  \mathbb{E}\left( \frac{|R_1|}{|R|}: \ R_1 \in C_j \mbox{ is horizontal}, \ d= d' \right) &= \frac{1}{|R|} \int_0^L |R_1(p_1, \vc e_{d'})| 1_{R_1 \in C_j} \frac{d p_1}{L}\\
   &= \frac{1}{L^2} \int_0^L  |R_1(p_1,\vc e_{d'})| 1_{R_1 \in C_j} dp_1.
  \end{align*}
  We first discuss the case when $p_1$ is not close to $0$ or $L$ and the
  building block $\mathcal{B}$ is thus centered at $\vc p$ (see the definitions of
  $\delta_k^j$ and $\lambda_k^j$ in Algorithm \ref{ModelAmod}).
  Then if $\vc e_d= \vc e_1$, $\mathcal{B}=(p_1-\delta^{-1}/2, p_1+\delta^{-1}/2)\times (0,1)$ and
  $R_1=(0,p_1-\delta^{-1}/2)\times(0,1)$.
  Similarly, if $\vc e_d= \vc e_2$, $\mathcal{B}=(p_1-\delta/2, p_1+\delta/2)\times (0,1)$ and
  $R_1=(0,p_1-\delta/2)\times(0,1)$.
  Thus, for $R_1$ to be in $C_j$, we need that either $p_1-\delta^{-1}/2 \in
  (\lambda^{-j-1/2}\delta^{-1}, \lambda^{-j+1/2} \delta^{-1})$ or $p_1-\delta/2 \in
  (\lambda^{-j-1/2}\delta^{-1}, \lambda^{-j+1/2} \delta^{-1})$, respectively.
  
  We remark that if $\mathcal{B}$ is not centered in $p_1$, it touches the right-boundary. Hence, the generated
  rectangle $R_1$ will only be shorter than it would be otherwise and we may
  hence bound from above by the previously derived formula.

  Introducing the new variables of integration $x_1=p_1-\delta^{-1}/2\leq L$ or
  $x_1=p_1-\delta/2 \leq L$, we may thus bound
  \begin{align}
  \label{eq:probab_jl}
  \begin{split}
    \frac{1}{L^2} \int_0^L  |R_1(p_1,e_d)| 1_{R_1 \in C_j} dp_1 &\leq \frac{1}{L^2} \int_{\lambda^{-j-1/2}\delta^{-1}}^L x_1 dx_1 \\
    &= \frac{x_1^2}{2L^2}|_{\lambda^{-j-1/2}\delta^{-1}}^L.
    \end{split}
  \end{align}
  We recall that by symmetry the volume fraction due to $R_2$ on the
  right-hand-side is of the same size and we can hence estimate the full volume
  fraction by:
  \begin{align*}
 & \mathbb{E}\left( \frac{|R_1|}{|R|}: \ R_1 \in C_j \mbox{ is horizontal} \right)+ \mathbb{E}\left( \frac{|R_2|}{|R|}: \ R_2 \in C_j \mbox{ is horizontal} \right)\\
  &\leq
    \frac{x_1^2}{L^2}|_{\lambda^{-j-1/2}\delta^{-1}}^L= 1- (\frac{\lambda^{-j-1/2}\delta^{-1}}{L})^2.
  \end{align*}
  We recall that $L \in (\lambda^{-j-1/2}\delta^{-1}, \lambda^{-j+1/2}
  \delta^{-1})$ and thus this is zero if $L$ is on the smaller end of the range and
  bounded above by
  \begin{align*}
    1-\lambda^{-2}.
  \end{align*}

  \underline{The case $l<j$:}
  We argue analogously as in the case $l=j$ except that the upper limit of the
  interval of integration in the analogue of \eqref{eq:probab_jl} is given by $\lambda^{-j+1/2}\delta^{-1} < L$ instead: Thus, in this case,
  \begin{align*}
   & \mathbb{E}\left( \frac{|R_1|}{|R|}: \ R_1 \in C_j \mbox{ is horizontal} \right)+ \mathbb{E}\left( \frac{|R_2|}{|R|}: \ R_2 \in C_j \mbox{ is horizontal} \right)\\
  &\leq
    \frac{x_1^2}{L^2}|_{\lambda^{-j-1/2}\delta^{-1}}^{\lambda^{-j+1/2}\delta^{-1}} = \frac{\lambda^{-2j+1}\delta^{-2} - \lambda^{-2j-1}\delta^{-2}}{L^2}.
  \end{align*}
  Since $L^2\geq \lambda^{-2l-1}\delta^{-2}$ and $l\leq j-1$ this can be
  estimated by
  \begin{align*}
    \frac{\lambda^{-2j+1} - \lambda^{-2j-1}}{\lambda^{-2l-1}}\leq \lambda^{-2j+2l} (1-\lambda^{-2}).
  \end{align*}
  This concludes the proof of the claim \eqref{eq:5} and thus of the lemma.
\end{proof}

It remains to prove Lemma \ref{lem:lowerbound}.
Due to the modification $(2)$ in the definition of Algorithm \ref{ModelAmod} we
here obtain a very short, straightforward proof.
Subsequently we discuss how to obtain similar results for Algorithm \ref{ModelA}
using more sophisticated methods.

\begin{proof}[Proof of Lemma \ref{lem:lowerbound}]
  We argue similarly as in the derivation of equation \eqref{eq:200} and claim
  that there exist constants such that we obtain the following \emph{upper} bound on
  the volume covered:
  \begin{align}
    \label{eq:21}
    \mathbb{E}(|\mathcal{V}^k|)-\mathbb{E}(|\mathcal{V}^{k+1}|) \leq \sum_{j} c_j \mathbb{E}(V^k_{j}),
  \end{align}
  for constants $c_j >0$ which are independent of $k$.
  Since $\sum_j \mathbb{E}(V_j^k)=\mathbb{E}(|\mathcal{V}^k|)$, a (possibly highly suboptimal) upper bound of the  right-hand-side is given by
  \begin{align*}
    \max(c_j) \mathbb{E}(|\mathcal{V}^k|).
  \end{align*}
  It hence follows that
  \begin{align*}
    \mathbb{E}(|\mathcal{V}^{k+1}|) \geq (1-\max(c_j)) \mathbb{E}(|\mathcal{V}^k|).
  \end{align*}
  We now claim that due to the second modification in the definition of
  Algorithm \ref{ModelAmod} it holds that $\max(c_j)\leq 0.3$ and the result
  hence follows.

  In order to compute the constants $c_j$, we again individually consider each
  rectangle $R \in \mathcal{V}^k$ and after rescaling and possibly rotating by
  $\frac{\pi}{2}$ may assume that
  \begin{align*}
    R= (0,L) \times (0,1)
  \end{align*}
  with $L \geq 1$.
  If we pick the vertical direction, $\vc e_d= \vc e_2$, (which occurs with probability $\frac{1}{2}$),
  we insert a translate of $(0,\delta)\times(0,1)$ and hence cover a very small
  fraction
  \begin{align*}
  \frac{\delta}{L}\leq \delta.
  \end{align*}

  If we instead pick the horizontal direction, $\vc e_d= \vc e_1$, (which also occurs with probability
  $\frac{1}{2}$), we expect to cover a larger volume fraction of $R$.
  We distinguish three cases:
  \begin{itemize}
  \item 
  If $L<\delta^{-1}/2$, the inserted rectangle is a translate of $(0,L)\times
  (0,L \delta)$ and hence covers a volume fraction
  \[L\delta \leq \frac{1}{2}.\]
  See Figure \ref{fig:Amod} top right.
\item 
  Similarly, if $L\geq 2 \delta^{-1}$, we insert a translate of 
  $(0,\delta^{-1})\times(0,1)$ and thus cover a volume fraction
  \[\frac{\delta^{-1}}{L}\leq \frac{1}{2}.\]
  See Figure \ref{fig:Amod} center right.
\item 
  Finally, if $\frac{\delta^{-1}}{2}\leq L < 2 \delta^{-1}$, we are in the case
  $(2)$ of Algorithm \ref{ModelAmod} (see Figure \ref{fig:Amod} bottom).
  As we only modify $R$ inside one quadrant we cover at most
  \[\frac{1}{4}\]
  of the volume.
\end{itemize}
Thus, in all theses cases for $\vc e_d= \vc e_1$ we cover at most $\frac{1}{2}$ of the
volume.

Combining the estimates for $\vc e_d= \vc e_2$ and $\vc e_d= \vc e_1$ (each with probability
$1/2$), then yields the bound
  \begin{align*}
  c_j \leq \frac{1}{2} \delta + \frac{1}{2} \frac{1}{2} = \frac{1}{2} \delta + \frac{1}{4} \leq 0.3,
  \end{align*}
  provided $\delta\leq 0.1$.
\end{proof}

\begin{rmk}
\label{rmk:improve}
Finally, let us briefly comment on some additional challenges in carrying out
the tail estimates \eqref{eq:300} \emph{without} the quadrant modification $(2)$ in Algorithm
\ref{ModelAmod}.
Consider a rectangle  $R \in \mathcal{V}^k$, $R=(0,L)\times(0,1)$ with
$\frac{1}{2}\delta^{-1} \leq L \leq 2 \delta^{-1}$ and suppose we picked
$\vc e_d= \vc e_1$.
Then, if $\frac{1}{2}\delta^{-1}\leq L\leq \delta^{-1}$, we insert a translate of $(0,L)\times(0,L \delta)$
and cover a fraction $L\delta \in [\frac{1}{2},1]$.
Similarly, if $\delta^{-1}\leq L\leq 2 \delta^{-1}$, we insert a translate of
$(0,\delta^{-1})\times (0,1)$ and cover a fraction $\frac{\delta^{-1}}{L} \in
[\frac{1}{2},1]$.
Therefore, the best naive upper bound for $\max(c_j)$ as in the proof of Lemma
\ref{lem:gain} we can achieve is given by
\begin{align*}
  \frac{1}{2} \delta + \frac{1}{2} 1, 
\end{align*}
and hence
\begin{align*}
  \mathbb{E}(|\mathcal{V}^{k+1}|)\geq (1-\frac{1}{2} \delta - \frac{1}{2} 1) \mathbb{E}(|\mathcal{V}^{k}|)= (0.5-\frac{1}{2} \delta)  \mathbb{E}(|\mathcal{V}^{k}|).
\end{align*}
Unlike the factor $0.7$ obtained in Lemma \ref{lem:lowerbound} this estimate is
not sufficient to close the inductive argument for \eqref{eq:300}.

In order to improve this bound, we thus need to exploit that the estimate in
terms of $\max(c_j)$ is very rough and not actually attained.
Indeed, we may employ an approach similar to the one of Lemma
\ref{lem:gain} to show that
\begin{align*}
  \mathbb{E}(V^{k}_0) \leq \theta \ \mathbb{E}(|\mathcal{V}^{k}|),
\end{align*}
for an explicit constant $\theta \in (0,1)$.
That is, only some part of the total volume is covered by rectangles $R \in
C_0$.
Then instead of bounding by $\max(c_j)$ we may use
\begin{align*}
  \theta (\frac{1}{2} \delta + \frac{1}{2} 1) + (1-\theta) \max_{j\neq 0} c_j.
\end{align*}
Unfortunately, while these and further improvements allow us to deduce that
$\mathbb{E}(|\mathcal{V}^{k+1}|)\geq 0.6\ \mathbb{E}(|\mathcal{V}^{k}|)$, this
still is not sufficient to close the inductive estimate \eqref{eq:300}. We thus
opted to simplify discussions by considering the modified Algorithm \ref{ModelAmod}.
\end{rmk}

\section{Simulations}
\label{sec:sim}

In this final section, we discuss the numerical implementation of our models from Algorithms \ref{ModelA} and \ref{ModelB} and compare it to the results in \cite{BCH15} and \cite{TIVP17} on the one hand and to our simulations from \cite{RTZ19} on the other hand. 

\subsection{Strategy}
\label{sec:sim_strat}

In order to run the simulations, we perform the two following simplifications to our models, which significantly reduce the computational cost of the algorithms.\\

\paragraph{\emph{Change 1:}} For Model A we change the definition of the sets $\mathcal{B}_k^j$ whenever $\delta\ell_{d_k^j}^j\geq \ell_{(d_k^j)^\perp}^j$ (this is the degenerate case). In this case we define $\alpha_k^j:= \frac{\ell_{(d_k^j)^\perp}^j}{ \ell_{d_k^j}^j}$, $N_k^j = \left\lfloor\frac{\delta}{\alpha_k^j}\right\rfloor$ and 
$$
\mathcal{B}_k^j:= \bigl\{\vc x \in D_k^j \colon \vc x\cdot \vc e_{d_k^j} \in (\vc p_k^j\cdot \vc e_{d_k^j} - \delta_k^j N_k^j\ell_{(d_k^j)^\perp}^j,\vc p_k^j\cdot \vc e_{d_k^j} + (1-\delta_k^j)N_k^j\ell_{(d_k^j)^\perp}^j)\bigr\},
$$
where 
\begin{align*}
\delta_k^j&:=\argmin\left\{\left|s-\frac12\right|\colon s\in(0,1) \right.\\
& \left. \qquad  \qquad \text{ and both }\vc p_k^j- s N_k^j\ell_{(d_k^j)^\perp}^j\vc e_{d_k^j},\vc p_k^j + (1-s)N_k^j\ell_{(d_k^j)^\perp}^j\vc e_{d_k^j} \in D_k^j	\right\}.
\end{align*}
See Figure \ref{fig:Ncopies} for an illustration.
Similarly, for Model B we set $\alpha_k:= \frac{\ell_{(d_k)^\perp}}{ \ell_{d_k}}$, $N_k = \left\lfloor\frac{\delta}{\alpha_k}\right\rfloor$ and 
$$
\mathcal{B}_k:= \bigl\{\vc x \in \mathcal{C}(\mathcal{V}_{k-1},\vc p_k) \colon \vc x\cdot \vc e_{d_k} \in (\vc p_k\cdot \vc e_{d_k} - \delta_k N_k\ell_{(d_k)^\perp},\vc p_k\cdot \vc e_{d_k} + (1-\delta_k)N_k\ell_{(d_k )^\perp})\bigr\},
$$
whenever $\delta\ell_{d_k}\geq \ell_{(d_k)^\perp}$. Again, here
\begin{align*}
\delta_k&:=\argmin\left\{\left|s-\frac12\right|\colon s\in(0,1) \right. \\
& \left. \qquad \qquad \text{ and both }\vc p_k- s N_k\ell_{(d_k)^\perp}\vc e_{d_k},\vc p_k + (1-s)N_k\ell_{(d_k)^\perp}\vc e_{d_k} \in \mathcal{C}(\mathcal{V}_{k-1},\vc p_k)	\right\}.
\end{align*}

We remark that, in both cases, one can repeat exactly the same proofs as in Section \ref{sec:conv-algor}, where the only difference here is that 
$$
\tilde{c}_A := \max\left\{p + (1-p)\max\{2^{-1},(1-\delta)\},(1-p) + p\max\{2^{-1},(1-\delta)\}\right\}\in(0,1).
$$
That means, this change deteriorates the rate of convergence in the case that $\delta>\frac12.$  
\medskip

\begin{figure}[htb]
  \centering
  \includegraphics[width=\linewidth,page=18]{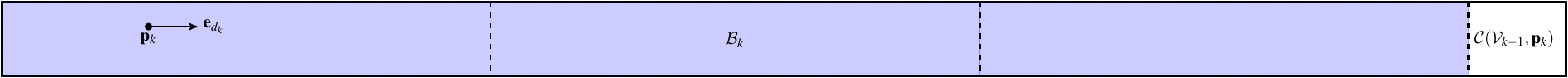}
  \caption{If $\delta l^j_{d_k^j}\geq l^{j}_{(d_k^j)^\perp}$, that is, in the
    degenerate case, in the Algorithms \ref{ModelA} and \ref{ModelB} we
    prescribed that we cover all of $D_k^j$. In order to simplify the numerical
    implementation we modify our construction to instead insert a maximal number of
    copies of the building blocks (in this example $N_k^j=3$).}
  \label{fig:Ncopies}
\end{figure}
\paragraph{\emph{Change 2:}} 
We define
$$
\Omega_1 = (0,1)\times(0,\delta),\qquad \Omega_2 = (0,\delta)\times(0,1),
$$
and we construct, following \cite{DPR}, two solutions $\vc z_1,\vc z_2\in W^{1,\infty}(\Omega_i;\R^2)\cap W^{1+s,p}(\Omega_i;\R^2)$ to 
\begin{align*}
\nabla \vc u &\in K \mbox{ a.e. in } \Omega_i, \\
\vc u & = \mt M \vc x \mbox{ on } \partial \Omega_i, 
\end{align*}
for $i=1,2$, and where $K$ is given by \eqref{defK_intro} and $\mt M\in K^{qc}$. Here, as in Theorem \ref{ThmCI}, $\theta_0\in(0,1)$ and $(s,p)\in [0,1)\times(1,\infty)$ are arbitrary and such that $sp<\theta_0$. Then, every time we have to make a replacement construction in the rectangle $\mathcal{B}$ (that is in one of the rectangles $\mathcal{B}_k^j$ in Model A, or in one of the rectangles $\mathcal{B}_k$ for Model B) we argue as follow:
\begin{itemize}
\item[(i)] if $\mathcal{B}=\vc c_0 + \lambda\Omega_i$ for some $\vc c_0 \in \R^2$, $\lambda\in(0,1]$, then we set $\vc y_k(\vc x) = \lambda  \vc z_i\left(\frac{\vc x-\vc c_0}{\lambda} \right)$ in $\mathcal{B}$. Indeed, we remark that, in $\mathcal{B}$, $\vc z_{\mathcal{B}}(\vc x):= \lambda \vc z_i\left(\frac{\vc x-\vc c_0}{\lambda} \right)$ satisfies $\nabla \vc z_{\mathcal{B}}\in K$ a.e., $\vc z_{\mathcal{B}}(\vc x) = \mt M\vc x$ on $\partial\Omega$. Therefore, the convergence of the model to the desired limiting stress free deformation is not affected by this change.

\item[(ii)] if $\mathcal{B}\neq \vc c_0 + \lambda\Omega_i$ for any $\vc c_0\in\R^2,\lambda\in(0,1]$, according to the changes to the model in the above paragraph, we have the existence of $\vc c_0\in\R^2,\lambda\in(0,1], N\in\mathbb{N}$ such that $\mathcal{B} = \bigcup_{n=0}^{N-1}\left(\vc c_0 + n\lambda \vc e_i + \lambda\Omega_i\right)$. In this case, as in the above one, we set $\vc y_k(\vc x) = \lambda \vc z_i\left(\frac{\vc x-\vc c_0- n\lambda \vc e_i}{\lambda} \right)$ for any $\vc c_0 + n\lambda \vc e_i + \lambda\Omega_i$ and $n=0,\dots,N-1.$ Again the convergence of the algorithm to the desired limiting stress free deformation is not affected by this change.
\end{itemize}
Regarding the regularity of the resulting microstructure, as in the proof of Theorem \ref{RegA}, we have to ensure that the modification of our construction still satisfies an estimate of the form 
\begin{align*}
\int\limits_{\Omega} \int\limits_{\Omega} \frac{|\nabla \mt v_k(\vc x)- \nabla \mt v_k(\vc y)|^p}{|\vc x - \vc y|^{2+sp}}\mathrm{d}\vc x \mathrm{d}\vc y 
&\leq \sum\limits_{j} \int\limits_{\mathcal{B}_k^j} \int\limits_{\mathcal{B}_k^j}  \frac{|\nabla \mt v_k(\vc x)-\nabla \mt v_k( \vc y)|^p}{|\vc x - \vc y|^{2+sp}} \mathrm{d}\vc x \mathrm{d}\vc y \\
& \quad + 2 \sum\limits_{j} \int\limits_{\mathcal{B}_k^j} \int\limits_{(\mathcal{B}_k^j)^c}  \frac{|\nabla\mt v_k(\vc x)- \nabla \mt v_k(\vc y)|^p}{| \vc x -  \vc y|^{2+sp}}\mathrm{d}\vc x \mathrm{d} \vc y \\
& \leq c \sum\limits_{j}\Per(\mathcal{B}_k^j)|\mathcal{B}_k^j|^{1-sp},
\end{align*} 
where, in our modified construction, we have to replace the old building blocks $\mathcal{B}_k^j$ by the blocks $\mathcal{B}$ described above. Since the second contribution is estimated ``generically'', not using properties of $\mt v_k$ (see the proof of Theorem \ref{RegA}), it suffices to discuss contributions of the form
\begin{align*}
 \int\limits_{\mathcal{B}} \int\limits_{\mathcal{B}}  \frac{|\nabla \mt v_k(\vc x)-\nabla \mt v_k( \vc y)|^p}{|\vc x - \vc y|^{2+sp}} \mathrm{d}\vc x \mathrm{d}\vc y .
\end{align*}

To this end, we consider the two cases (i), (ii) described above:
First, by stacking $N$ blocks of the microstructures on top of each other (that means in case $\mathcal{B} = \bigcup_{n=0}^{N-1}\left(\vc c_0 + n\lambda \vc e_i + \lambda\Omega_i\right)=:\bigcup\limits_{n=1}^{N-1} \Omega_i^n$), we have that for $\mt v_k:= \vc y_{k+1} - \vc y_k$ (cf. proof of Theorem \ref{RegA})
\begin{align*}
&\int_{\mathcal{B}}\int_{\mathcal{B}}\frac{|\nabla \mt v_{k}(\vc x) - \nabla \mt v_{k}(\hat{\vc x})|^p}{|\vc x - \hat{\vc x}|^{2+sp}}\,\mathrm{d}\vc x\,\mathrm{d}\hat{\vc x} \\
& = \sum\limits_{n=0}^{N-1} \int\limits_{\Omega_i^n}\int\limits_{\Omega_i^n}\frac{|\nabla \mt v_{k}(\vc x) - \mt v_{k}(\hat{\vc x})|^p}{|\vc x - \hat{\vc x}|^{2+sp}}\,\mathrm{d}\vc x\,\mathrm{d}\hat{\vc x} 
+ \sum\limits_{n=0}^{N-1}  \int\limits_{\mathcal{B}\setminus\Omega_i^n} \int\limits_{\Omega_i^n}\frac{|\nabla \mt v_{k}(\vc x) - \nabla \mt v_{k}(\hat{\vc x})|^p}{|\vc x - \hat{\vc x}|^{2+sp}}\,\mathrm{d}\vc x\,\mathrm{d}\hat{\vc x}\\
&\leq N \left|\nabla \vc z_i\left(\frac{\vc x-\vc c_0}{\lambda} \right)\right|^p_{\dot{W}^{s,p}(\vc c_0 + \lambda\Omega_i)} +c N \Per(\Omega_i^n)^{sp} |\Omega_i^n|^{1-sp}\\
&\leq N \left|\nabla \vc z_i\left(\frac{\vc x-\vc c_0}{\lambda} \right)\right|^p_{\dot{W}^{s,p}(\vc c_0 + \lambda\Omega_i)}
 + C\left( N \alpha \ell_d\right)^{sp}(|\mathcal{B}|)^{1-sp}\\
&\leq N \left|\nabla \vc z_i\left(\frac{\vc x-\vc c_0}{\lambda} \right)\right|^p_{\dot{W}^{s,p}(\vc c_0 + \lambda\Omega_i)} + C \left(\Per(\mathcal{B})\right)^{sp}(|\mathcal{B}|)^{1-sp}.
\end{align*}

But since 
\begin{equation}
\label{Scaling}
\begin{split}
&\int_{\lambda\Omega_i}\int_{\lambda\Omega_i}\frac{|\nabla \vc z_i (\lambda^{-1}\vc x) - \nabla \vc z_i (\lambda^{-1}\hat{\vc x})|^p}{|\vc x - \hat{\vc x}|^{2+sp}}\,\mathrm{d}\vc x\,\mathrm{d}\hat{\vc x} \\
& = \lambda^{2-sp}\int_{\Omega_i}\int_{\Omega_i}\frac{|\nabla \vc z_i (\vc x) - \nabla \vc z_i (\hat{\vc x})|^p}{|\vc x - \hat{\vc x}|^{2+sp}}\,\mathrm{d}\vc x\,\mathrm{d}\hat{\vc x}
\leq 
\lambda^{2-sp}\left|\nabla \vc z_i\right|^p_{W^{s,p}(\Omega_i)},
\end{split}
\end{equation}
and since in this case 
$$
N\lambda^{2-sp} = (N\lambda)^{sp}(N\lambda^2)^{1-sp}\leq c\left(\Per(\mathcal{B})\right)^{sp} |\mathcal{B}|^{1-sp},
$$ 
we obtain 
\begin{align*}
&\int_{\mathcal{B}}\int_{\mathcal{B}}\frac{|\nabla \mt v_k(\vc x) - \nabla \mt v_k (\hat{\vc x})|^p}{|\vc x - \hat{\vc x}|^{2+sp}}\,\mathrm{d}\vc x\,\mathrm{d}\hat{\vc x} \\
&\leq c(\Per(\mathcal{B}))^{sp}|\mathcal{B}|^{1-sp} \left|\nabla \vc z_i\right|^p_{W^{s,p}(\Omega_i)} + \left(\Per(\mathcal{B})\right)^{sp} |\mathcal{B}|^{1-sp}.
\end{align*}
We now notice that, by \eqref{Scaling}, also when $\mathcal{B}= \vc c_0 + \lambda\Omega_i$ we have 
$$
\left|\nabla \mt v_k\right|^p_{W^{s,p}(\mathcal{B})} \leq \lambda^{2-sp} \left|\nabla \vc z_i\right|^p_{W^{s,p}(\Omega_i)} \leq c|\mathcal{B}|^{1-sp}\left(\Per(\mathcal{B})\right)^{sp} \left|\nabla \vc z_i\right|^p_{W^{s,p}(\Omega_i)} .
$$
Therefore, we still infer \eqref{stima big} in the proof of Theorem \ref{RegA}, and hence we deduce the same regularity result as in Section \ref{sec:reg_sol} above also under this implementation of the model.

The outlined modifications of the algorithms thus have the \emph{computational advantage} that in the degenerate case it suffices to have the two ``standard'' convex integration solutions $\vc z_1, \vc z_2$ which can be inserted into the covering instead of having to produce new convex integration solutions for each aspect ratio. This is a substantial numerical improvement, since the production of the convex integration building blocks is the computationally most expensive part in our simulations. As explained this change does \emph{not} effect major changes in our theoretical regularity estimates for the solutions.

\subsection{Output of the simulations and comparison}

In this section, we present some of the output of our simulations of the described Algorithms \ref{ModelA} and \ref{ModelB}. 
As explained above, this combines a random covering by horizontal or vertical rectangles of a fixed width-to-length ratio $\delta>0$ and the filling of these by the convex integration building blocks from Theorem \ref{ThmCI}. For given boundary data $\mt M\approx \begin{pmatrix}
      0.939 & 0\\
      0 & 1.064
    \end{pmatrix}$ and $\gamma=0.5$,
     the horizontal and vertical convex integration building blocks are illustrated in Figures \ref{fig:buildingblocks} and \ref{fig:buildingblocks1}, respectively.

\begin{figure}[t]
\includegraphics[height= 4 cm]{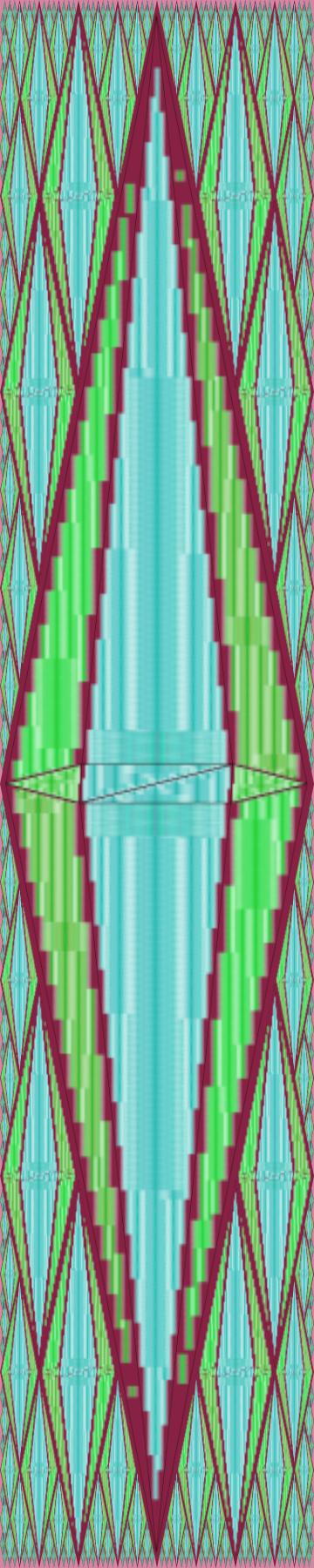}
\hspace{0.5 cm}
\includegraphics[height = 4cm]{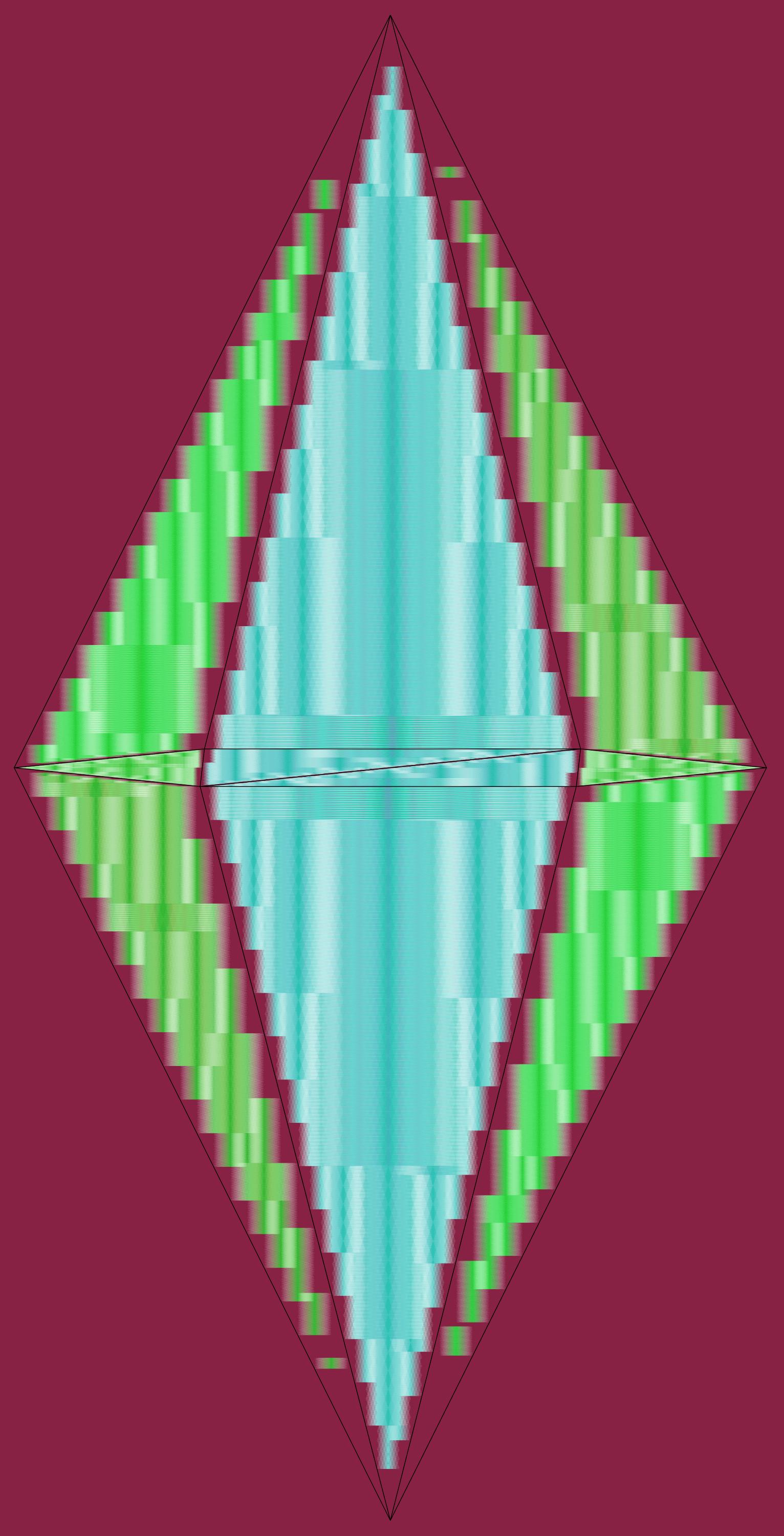}
\caption{The vertical building blocks: The thin needle structures in their actually used size (left) and a blown-up version of this (right). The illustration shows the building block after several iterations of the convex integration scheme. }
\label{fig:buildingblocks1}
\end{figure}

For the Algorithm \ref{ModelA} this yields structures as depicted in Figure \ref{fig:ModelA} left. Here fractal structures emerge similarly as in \cite{BCH15, CH18} and \cite{TIVP17} with the main difference that the use of convex integration building blocks results in a \emph{stress-free} solution in the limit $k \rightarrow \infty$ (which thus involves further fine length scales in the building blocks from Theorem \ref{ThmCI}).

\begin{figure}[ht]
\includegraphics[width=0.4 \textwidth]{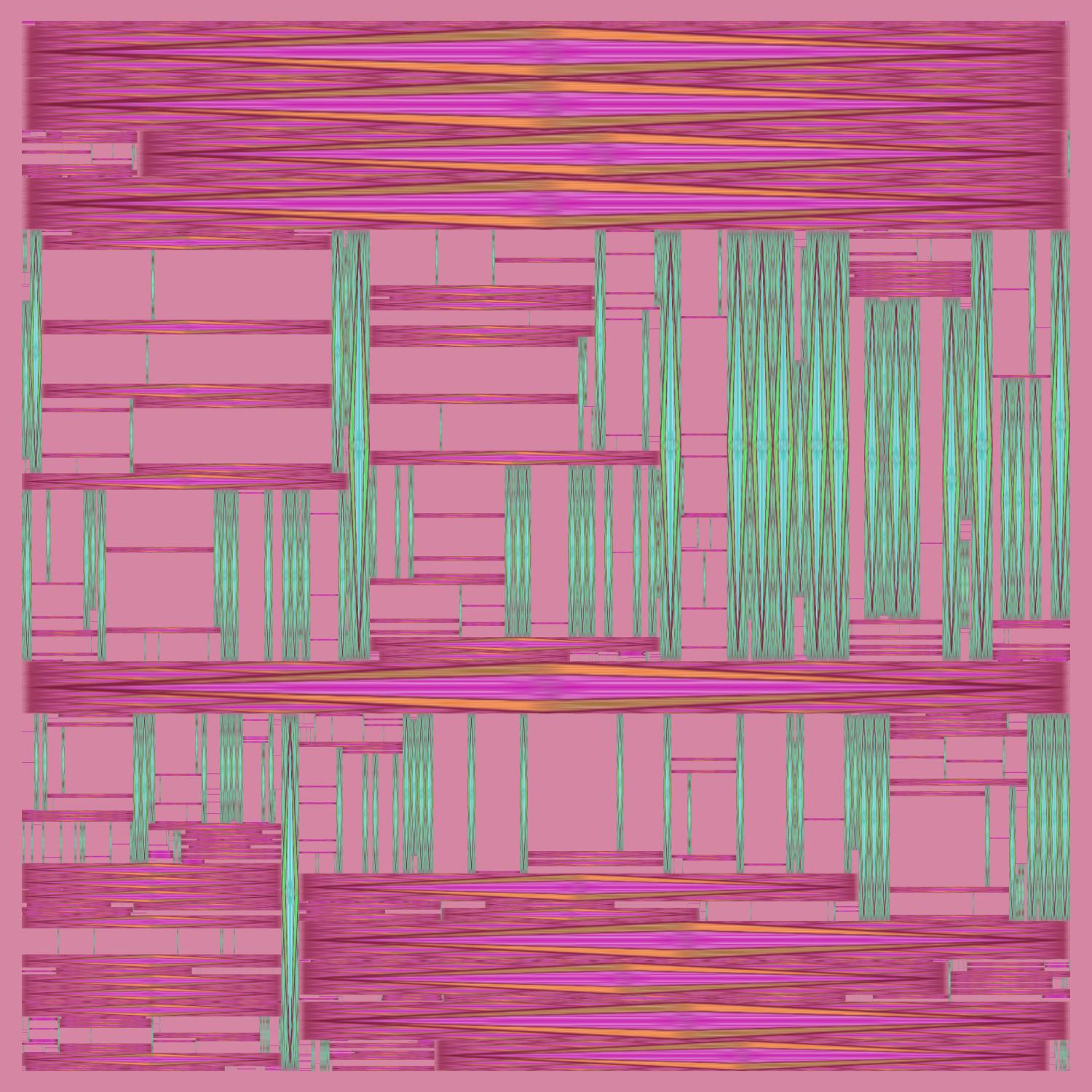}
\includegraphics[width=0.4 \textwidth]{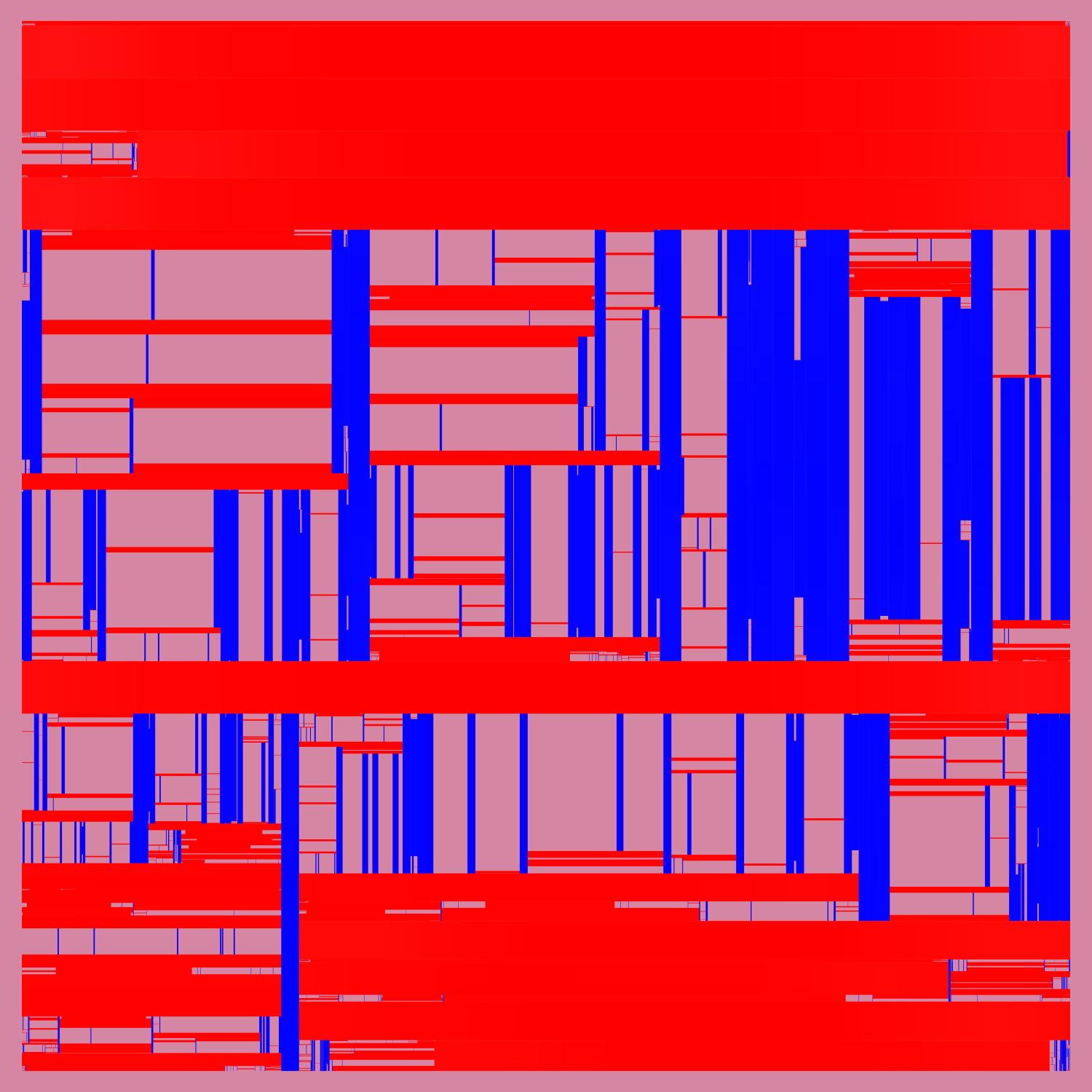}
 \caption{The random convex integration solution produced by Algorithm
   \ref{ModelA} (left) and a random packing without interior structure generated
   by the same random covering arguments as in Algorithm \ref{ModelA} (right)
   for $M\approx \diag(0.939, 1.064)$, $\gamma = 0.5$. 
By using the building blocks from Theorem \ref{ThmCI}, the structures which are obtained in the limit $k \rightarrow \infty$ of the Algorithm \ref{ModelA} become exactly stress-free solutions to the differential inclusion \eqref{eq:diff_incl1}. The illustration shows the microstructure after 11 iterations of the covering procedure. Thus, we have inserted roughly 2000 building block structures according to the iteration rules of Algorithm \ref{ModelA}. }
\label{fig:ModelA}
\end{figure}

For the Algorithm \ref{ModelB} we also obtain highly fractal structures (see Figure \ref{fig:ModelB}). As already observed in \cite{CH18}, after inserting the same number of rectangles, these however are more homogeneous than the ones from Algorithm \ref{ModelB}. Compared to the illustrations in \cite{CH18} this is still strongly observable but possibly slightly less pronounced in our illustrations than in \cite{CH18} due to the presence of a finite width.

\begin{figure}[ht]
\includegraphics[width=0.4 \textwidth]{StochCoverB.jpg}
\includegraphics[width=0.4 \textwidth]{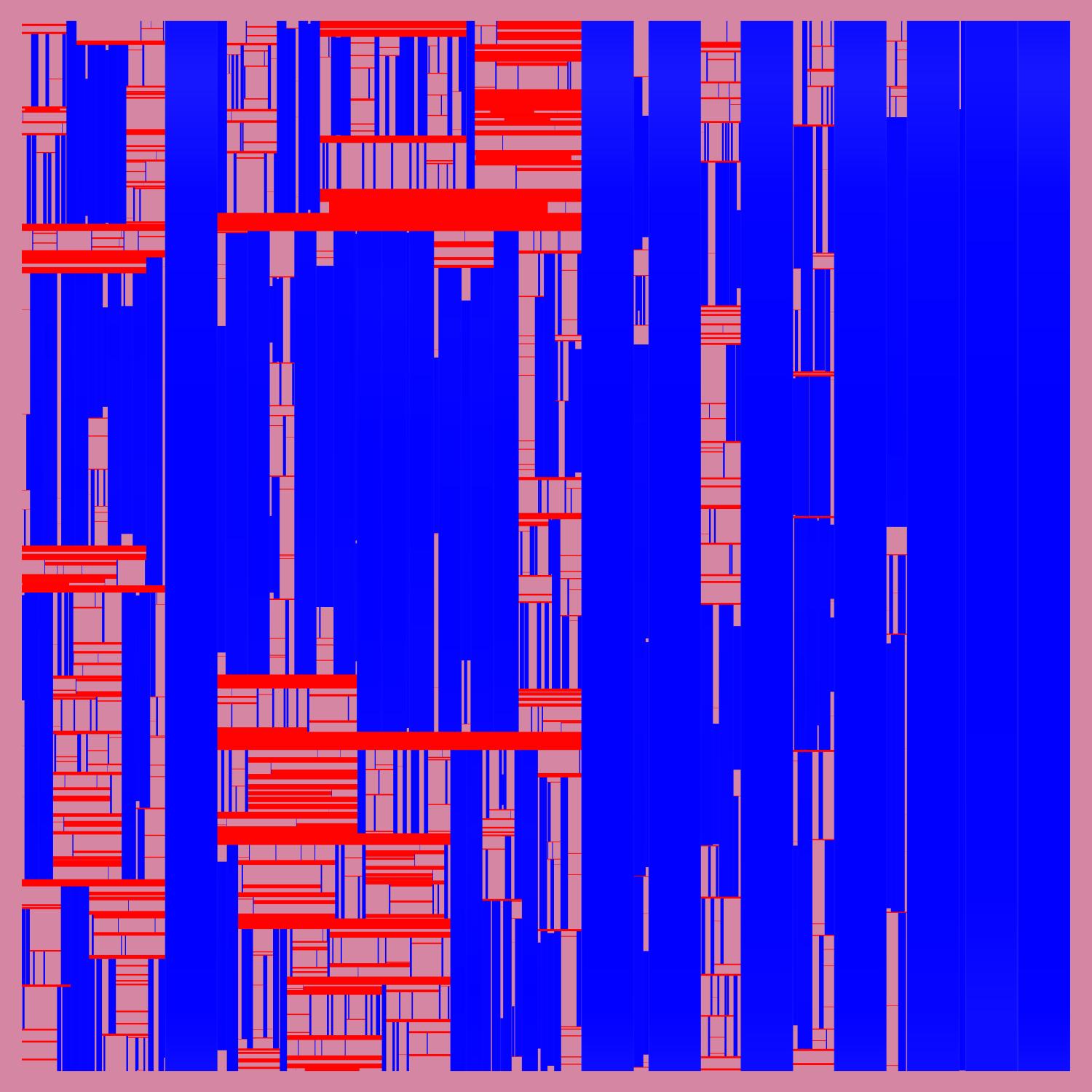}
\caption{The random convex integration solution for the boundary data 
$M\approx \diag(0.939, 1.064)$, $\gamma = 0.5$.
   produced by Algorithm \ref{ModelB} (left) and a random
   packing without interior structure generated by the same random covering
   arguments as in Algorithm \ref{ModelB} (right).
  As in the setting of the Algorithm \ref{ModelA} the fact that we rely on building blocks with convex integration structure implies that in the limit $k \rightarrow \infty$ the deformations are exactly stress-free solutions to the differential inclusion \eqref{eq:diff_incl1}. In the illustration here we have iterated the algorithm roughly 1000 times and have thus introduced roughly 1000 covering rectangles. Due to the iteration scheme of Algorithm \ref{ModelB} the covering boxes are distributed much more uniformly than in Algorithm \ref{ModelA} and, on average, cover a larger volume fraction of the domain after the same number of boxes have been introduced.}
\label{fig:ModelB}
\end{figure}

\subsection{Length scale statistics obtained in the Algorithms \ref{ModelA} and \ref{ModelB}}
\label{sec:sim_impl}

In order to eventually compare our results to the experimental data (see for instance \cite{VOMLRPP94,CMOPV98}, where universal exponents are obtained for each phase transformation), below we present length scale statistics of our solutions after a finite number of iterations of our algorithms. Here as a measure of the lengths we consider the long side of the rhombi-constructions as a respective measure. The distribution of the lengths involved in the random covering is analogous to the ones obtained in \cite{CH18} or \cite{TIVP17}. Further, finer length scales are however involved in the individual rhombi-building blocks (which themselves are obtained through iterative algorithms, see for instance \cite{DPR} or \cite{RZZ18}, which here are illustrated up to a third order iteration). Due to the use of the ``infinite iteration building blocks'' in the rhombi-constructions (see \cite{RTZ19} for this notation), we however remark that while these statistics may eventually serve as a comparison to the experimental data, they are not directly linked to the regularity exponents of the convex integration solutions as in the case of the ``finite convex iteration building blocks'' which had been discussed in \cite{RTZ19}.

\subsection{Combined length scale distributions}
\label{sec:distr}
A quantity that is of considerable experimental interest is the number and strength of acoustic emissions during the nucleation process \cite{PMLV13, VOMLRPP94}. It is believed that this is related to the length scale distribution of the microstructure which emerges upon nucleation \cite{CH18}. In order to eventually allow for comparisons of our theoretic findings with experimental results, we also analyze this quantity and present some numerical experiments on its computation.

In our algorithms essentially two length scale distributions enter in the computation of the overall length scale distribution: On the one hand, we consider the lengths scales of the outer random packing (this essentially has the same distribution as the length scale distribution from \cite{TIVP17}). On the other hand these also have an internal length scale distribution, since there is structure also within our building blocks. The inner structure in turn is again determined by two lengths scale distributions which consist of a covering of the given domain by model rectangles and the covering of the model rectangles by the rhombi-constructions.

A heuristic computation shows that in general the length scale distribution will be given by a competition of the involved length scales. We next discuss this for the case of two nested scales: Let us assume that for $y \in (0,1)$ the length scale distribution of the outer blocks is described by a function $f(y)$ which counts the number of blocks (of ratio $\delta$) of size $y$ (in this back-of-the-envelope calculation we exclude the degenerate setting). Let us further suppose that each block of the size $(0,1)\times (0,\delta)$ has an interior length scale distribution modelled by a function $g(x)$ which assigns to every structure of length scale $x$ the number $g(x)$ of such scales (again we assume that there is a certain non-degeneracy of our structures here). Thus, for every fixed outer structure of the size $(0,y)\times (0,\delta y)$ (up to translation), by scaling, the inner structure has a length scale distribution given by $g(\frac{x}{y})$.

As a consequence, the overall number of structures of lengths $x$ are roughly given by $\int\limits_{0}^1 f(y) g(\frac{x}{y}) dy$. If both distributions $f, g$ are power laws, e.g. $f(y) = y^{\alpha}$, $g(x) = x^{\beta}$, we thus infer that
\begin{align}
\label{eq:lengths}
\int\limits_{0}^1 f(y) g\left(\frac{x}{y} \right) dy
= c(\alpha,\beta) x^{\beta}(1- x^{\alpha+1-\beta}) \sim \max\{x^{\alpha + 1}, x^{\beta}\}.
\end{align}

Numerically, we observe that, indeed, at least at our finite numerical resolution of the random convex integration scheme, the length scales can be well-approximated by power laws: Considering the boundary data
\begin{align*}
\mt M \approx \begin{pmatrix}
      0.939 & 0\\
      0 & 1.064
    \end{pmatrix},
\end{align*}
and $\gamma=0.5$ we obtain the following length scale distributions:
\begin{itemize}
\item[(i)] The histrogram of the length scale distribution inside a diamond building block (as in Figure \ref{fig:buildingblocks1}) is shown in Figure \ref{fig:histoinner}. A least square fit gives $g(x)\sim C x^{-2.107}$. 
\item[(ii)] Using this diamond we dyadically fill a rectangle as in Figure \ref{fig:dyadicpack}, which yields a length scale distribution $f(y)\sim C y^{-1}$ (see Lemma \ref{lem:packing} below for a more detailed argument). 
\item[(iii)] These rectangles then serve as building blocks in our stochastic packing, which itself exhibits length scale distribitions $h_A(z) \sim C z^{-1.49}$ and $h_B(z) \sim C z^{-1.48}$ (see Figure \ref{fig:histostoch}) for the Algorithms \ref{ModelA} and \ref{ModelB}, respectively.
\end{itemize}
The combination of the first two items (i), (ii) provides the overall length scale distribution of the building block constructions from Theorem \ref{ThmCI}.

Inserting the described distributions into (an iterated version of) \eqref{eq:lengths}, we 
 obtain an overall power law with exponent $\alpha = -2.107$. It is thus the length scale of the convex integration building blocks which dominate the overall length scales in our model.

\begin{figure}
\begin{center}
\begin{subfigure}[b]{0.45\textwidth}
\includegraphics[width=0.95\textwidth,page=20]{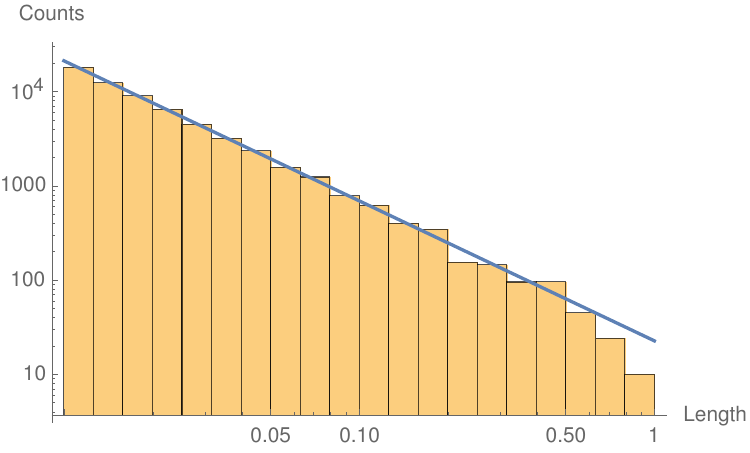}
\caption{Histogram of length scales from Algorithm \ref{ModelA}, with fitting $\frac{22.55}{x^{1.486}}$. An average volume fraction of approximately $0.914$ was covered over these realisations.}
\end{subfigure}
\begin{subfigure}[b]{0.45\textwidth}
\includegraphics[width=0.95\textwidth,page=21]{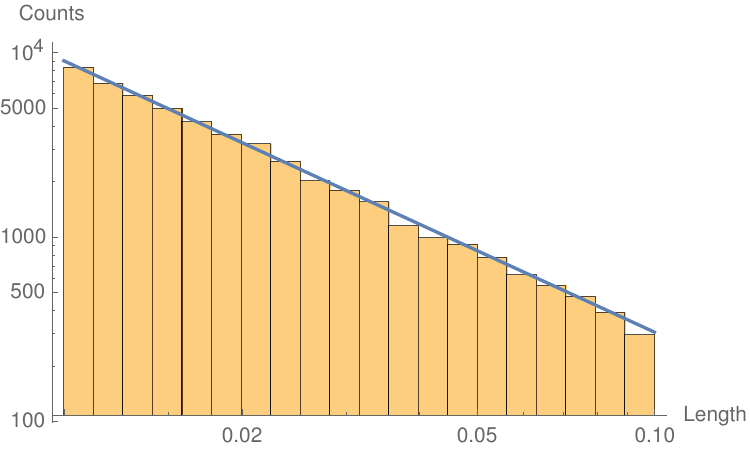}
\caption{Histogram of length scales from Algorithm \ref{ModelB}, with fitting $\frac{10.34}{x^{1.470}}$. An average volume fraction of approximately $0.926$ was covered over these realisations.}
\end{subfigure}
\end{center}
\caption{Histograms of length scales of inclusions from the Algorithms \ref{ModelA}, \ref{ModelB} on a rectangle of aspect ratio $0.05$. The particular implementation is exact, in the sense that the algorithm terminates only when it is impossible to generate an inclusion of length greater than $10^{-2}$, which is the range shown. In both cases, the histograms are generated over 10 realisations. Fitting curves are found by least squares regression of the histograms, using only the data range of $10^{-1}$ to $10^{-2}$ to avoid noise during the burn-in of the algorithm.}
\label{fig:histostoch}
\end{figure}

\begin{figure}
\centering
\includegraphics[width=0.6\textwidth,page=22]{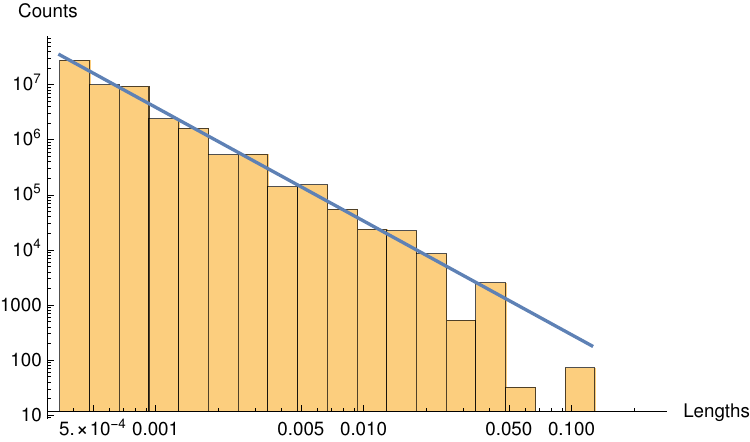}
\caption{Histogram of length scales for the convex integration algorithm using the matrix 
$M\approx \diag(0.939, 1.064)$, $\gamma = 0.5$.
This corresponds to the positive definite square root of $\mt C_1$ for the parameters $\lambda,\mu=0.1$ in the construction of \cite{DPR}. The histogram is shown with a fit of the form $\frac{2.47}{x^{2.107}}$, obtained by linear regression on the log-log histograms, using only data between $3.5\times 10^{-4}$ to $10^{-2}$ to avoid noise from burn-in. The data is exact on the data shown, in the sense that the implementation only terminates when all possible inclusions of length scale greater than $3.5\times 10^{-4}$ have been generated.}
\label{fig:histoinner}
\end{figure}

Although, in general, it is not yet proven that the functions $f(y)$ and $g(x)$ must be power laws, we conclude from the back-of-the-envelope computation from above, that the details of the interior structure and thus of the \emph{compatibility} requirement has an interesting, measureable impact on the experimentally measured length scale distributions in the described covering algorithms. In particular including compatibility thus provides important new and experimentally measurable information on the models of \cite{BCH15,CH18,TIVP17}.

\begin{appendix}
\section{A covering result}

Last but not least, for completeness, we discuss the covering result used in (ii) in Section \ref{sec:distr}.

\begin{lem}
\label{lem:packing}
Consider an axis-parallel rectangle $R$ of lengths $1:\delta$ and its greedy covering by dyadically rescaled copies of our diamond domain (illustrated in Figure \ref{fig:dyadicpack}).
Then for any $n\geq 1$ there are $8 \cdot 2^n$ diamonds with length scale $4^{-1}2^{-n+1}$.
\end{lem}

\begin{figure}
\centering
\includegraphics[width=0.5\textwidth,page=19]{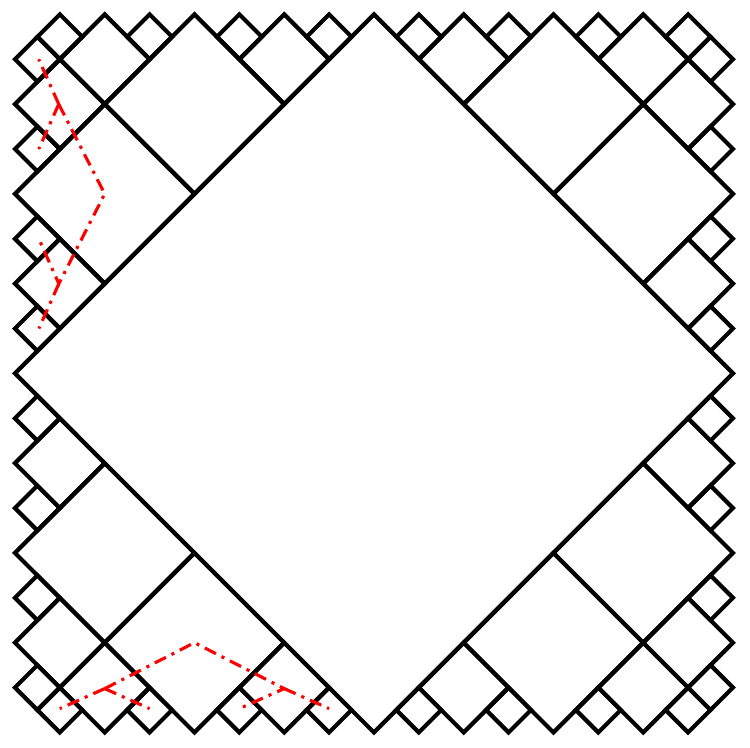} 
\caption{Example of the dyadic packing, showing two lineages.}
\label{fig:dyadicpack}
\end{figure}

\begin{proof}
We note that the largest diamond in the center has length scale $1$ and covers $\frac{1}{2}$ of the total volume.
In each of the four remaining regions we then insert two copies of diamond rescaled by a factor $\frac{1}{4}$, which corresponds to the case $n=1$.
As illustrated in Figure \ref{fig:dyadicpack} starting from each of these $8$ diamonds we obtain a tree-like structure of $2^n$ diamonds rescaled by a factor $\left(\frac{1}{2}\right)^n$, which concludes the proof.

We remark that by the same argument at step $N$ we have covered a total volume
\begin{align*}
\frac{1}{2} + \frac{1}{2}\sum_{n=1}^N 2^n 4^{-n} = 1 - 2^{-N-2},
\end{align*}
while the size of the boundary grows proportionally to $N$.
\end{proof}

\end{appendix}

\section*{Acknowledgements}

Francesco Della Porta and Angkana Rüland would like to thank the MPI MIS where part of this work was carried out. Angkana Rüland gratefully acknowledges funding by the Deutsche Forschungsgemeinschaft (DFG, German Research Foundation) -- project number RU 2049/1-1 within the SPP 2256 ``Variational Methods for Predicting Complex Phenomena in Engineering Structures and Materials''. She is a member of the Heidelberg STRUCTURES Cluster of Excellence (which is part of Germany’s Excellence Strategy, EXC-2181/1 - 390900948).
Jamie M. Taylor has been partially supported by the Basque Government through the BERC 2018-2021 program; and by Spanish Ministry of Economy and Competitiveness MINECO through BCAM Severo Ochoa excellence accreditation SEV-2017-0718 and through project MTM2017-82184-R funded by (AEI/FEDER, UE) and acronym ``DESFLU".
The work was partially written while Christian Zillinger was at BCAM. During that time his research was supported by the ERCEA under the grant 014 669689-HADE and also by the Basque Government through the BERC 2014-2017 program and by Spanish Ministry of Economy and Competitiveness MINECO: BCAM Severo Ochoa excellence accreditation SEV-2013-0323. Christian Zillinger gratefully acknowledges support by the DFG through the CRC 1173.

\bibliographystyle{alpha}
\bibliography{citations1}

\end{document}